\documentclass[pdflatex,sn-mathphys-num]{sn-jnl}

\usepackage{graphicx}
\usepackage{hyperref}
\usepackage{multirow}%
\usepackage{amsmath,amssymb,amsfonts}%
\usepackage{amsthm}%
\usepackage{mathrsfs}%
\usepackage[title]{appendix}%
\usepackage{xcolor}%
\usepackage{textcomp}%
\usepackage{manyfoot}%
\usepackage{booktabs}%
\usepackage{algorithm}%
\usepackage{algorithmicx}%
\usepackage{algpseudocode}%
\usepackage{listings}%

\theoremstyle{thmstyleone}%
\newtheorem{theorem}{Theorem}
\newtheorem{proposition}[theorem]{Proposition}%
\newtheorem{assumption}{Assumption}
\newtheorem{corollary}{Corollary}
\newtheorem{lemma}{Lemma}

\theoremstyle{thmstyletwo}%
\newtheorem{remark}{Remark}%

\theoremstyle{thmstylethree}%
\newtheorem{definition}{Definition}%

\newcommand{\FT}{\mathcal{F}}
\newcommand{\BT}{\mathcal{B}}
\newcommand{\UT}{\mathcal{U}}

\newcommand{\LT}{\mathcal{L}}

\usepackage{tcolorbox}
\newtcolorbox{mygraybox}{colframe = gray}

\begin{document}

\title[Nonlinear 
Optimal Recovery in Hilbert Spaces]{Nonlinear 
Optimal Recovery in Hilbert Spaces}


\author[1]{\fnm{Daozhe} \sur{Lin}}\email{dl3394@columbia.edu}

\author[1,2]{\fnm{Qiang} \sur{Du}}\email{qd2125@columbia.edu}

\affil[1]{\orgdiv{Department of Applied Physics and Applied Mathematics}, \orgname{Columbia University}, \orgaddress{\city{New York}, \state{NY},\postcode{10027}}}

\affil[2]{\orgdiv{Data Science Institute}, \orgname{Columbia University}, \orgaddress{\city{New York}, \state{NY},\postcode{10027}}}


\abstract{This paper investigates solution strategies for nonlinear problems in Hilbert spaces, such as nonlinear partial differential equations (PDEs) in Sobolev spaces, when only finite measurements are available. We formulate this as a \textbf{nonlinear optimal recovery} problem, establishing its well-posedness and proving its convergence to the true solution as the number of measurements increases. However, the resulting formulation might not have a finite-dimensional solution in general. We thus present a sufficient condition for the finite dimensionality of the solution, applicable to problems with well-defined point evaluation measurements. To address the broader setting, we introduce a \textbf{relaxed nonlinear optimal recovery} and provide a detailed convergence analysis. An illustrative example is given to demonstrate that our formulations and theoretical findings offer a comprehensive framework for solving nonlinear problems in infinite-dimensional spaces with limited data.}


\maketitle
\section{Introduction} 
Many mathematical models developed for the study and understanding of physical phenomena can be cast as finding solutions to linear or nonlinear problems in Banach spaces. These spaces could be infinite-dimensional, such as the case of nonlinear partial differential equations (PDEs). As analytical solutions in simple close-forms are often elusive or even nonexistent, 
numerical methods play a pivotal role in providing effective finite dimensional approximations and insights into the behavior of these complex systems. The subject of numerical PDEs, from finite difference and finite element methods to spectral methods, mesh-free techniques, and machine learning methods, continues to evolve, driven by advances in numerical algorithms and mathematical theory, as well as computational resources and practical considerations.
In applications, one of the frequently encountered challenges is the lack of a complete physical model or the limited availability of model parameters or data.
In such situations, we see the natural connection to optimal recovery problems.

Based on finite measurements, the optimal recovery aims to find the minimum norm solution that satisfies all the measurements. Recently, optimal recovery methods for solving linear PDEs have drawn considerable attention due to the finite-dimensional nature of the optimally recovered solutions, making them computationally tractable and amenable to efficient implementations \cite{chen2021solving, owhadi2019operator, binev2023solving}. Furthermore, the robust convergence theory associated with these methods underscores their appeal, offering a rigorous framework for analyzing and guaranteeing the accuracy of recovered solutions in various scientific and engineering applications \cite{chen2021solving,batlle2023error}.

\subsection{Summary of our Contributions} 
\label{subsec:contri}
The primary goal of this paper is to introduce a general framework for optimal recovery in abstract Hilbert spaces and to develop its mathematical theory under a suitable set of assumptions. We further discuss the optimality conditions and investigate scenarios in which our framework yields numerically solvable problems. For cases that might not appear directly tractable numerically, we present some generalizations as possible remedies. We also specialize to some concrete examples to illustrate our theory.

To be more specific, we present below a more formal description of the problems under consideration, together with a summary of our contributions.

\begin{itemize}
\item \textbf{An Abstract Nonlinear Problem with Finite measurements}: 
Given a nonlinear operator $\FT$ from a Hilbert space $\mathcal{U}$  
to a reflexive Banach space $\mathcal{B}$, we seek a solution $u$ in $\mathcal{U}$ to the following abstract nonlinear equation:
\begin{equation}\label{eq:nonleq}
    [\FT(u),\phi]=[f,\phi],\ \ \forall\phi\in\BT^*
\end{equation}
under the assumption that the only prescribed information on the data $f\in \mathcal{B}$ is through $N$ finite measurements,
 $\{[f,\phi_n]\}_{n=1}^N$,
for some given $\{\phi_n\in \mathcal{B}^*\}$, i.e., 
\begin{equation}\label{eq:fmeasure}
[\FT(u),\phi_n]=
[f,\phi_n],\ \ n=1,...,N.
\end{equation}
Here,
 $[\cdot,\cdot]$ denotes the duality pairing between $\mathcal{B}$ and $\mathcal{B}^*$, so that our constraint is defined naturally in the weak form with $\{\phi_n\}$ being test functions.
 
  We are interested in the problem for not only a given finite $N$ but also as $N\to \infty$,  assuming that $\text{Span}\{\phi_n\}_{n=1}^\infty$ forms a dense subset of $\mathcal{B}^*$. For easy representation, throughout this paper, we will assume that the abstract nonlinear equation admits a \textbf{unique solution} $u^*$.
 \footnote{Much of our discussions remain valid, with suitable modifications, for cases that the nonlinear problem has multiple isolated solutions, see comments given later.}
On the other hand,  the problem with only finite measurements \eqref{eq:fmeasure} is in general under-determined, thus its solution is not expected to be unique. This calls for further specifications to be illustrated later in Section \ref{sec:ORonNP}.

    \item \textbf{Nonlinear Optimal Recovery}: 
The nonlinear optimal recovery problem is a framework to construct a numerical solution to the above nonlinear problem with finite measuraments by solving the following variational problem:
\begin{equation}\label{eq:optrec}
\begin{aligned}
& \min_{u\in\mathcal{U}}\ \  \|u\|_\mathcal{U}\\
s.t. & \ \ [\FT(u),\phi_n]=[f,\phi_n],\ \ n=1,...,N
\end{aligned}
\end{equation}
In Theorem \ref{convergeNOR}, we demonstrate theoretically that, as $N\to \infty$, the solution $u_N$ of the nonlinear optimal recovery problem \eqref{eq:optrec} converges strongly to a weak solution $u^*$ of the abstract nonlinear problem \eqref{nonlinearproblem},
under suitable assumptions.

\item \textbf{Optimality and Finite Dimensionality Conditions}: 
Given a fixed $N$, for the solution $u_N$ of the optimal recovery problem \eqref{eq:optrec},  we present in Proposition \ref{OCNonlinear} the optimality condition 
\begin{equation}
    \label{eq:orcond}
u_N=\sum_{n=1}^N\lambda_n\psi(u_N,\phi_n).
\end{equation} 
Here, the operator
$\psi:\UT\times\BT^*\rightarrow\UT$ is defined by
\begin{equation}
    \label{eq:orpsi}
\langle \psi(u,\phi_n),v\rangle_\UT=[D_u\FT(u)v,\phi_n],\quad\forall
v\in \UT 
\ n=1,...,N,
\end{equation}
corresponding to the inner product  $\langle\cdot,\cdot\rangle_\UT$ of the Hilbert space $\UT$. 

In general, it might not be possible to characterize the solution $u_N$ in a finite-dimensional space. However, the following sufficient condition assures that the solution $u_N$ can be constructed with a finite-dimensional representation:
\begin{equation}
    \label{eq:fdcond}
    \psi(u,\phi_n)=\sum_{l=1}^{L}c_{l}(u,\phi_n)\psi_{l}(\phi_n),\qquad n=1,...,N, \; \forall u \in \UT,
    \end{equation}
for $L$ scalar functions $\{c_{l}:\UT\times\BT^*\rightarrow\mathbb{R}\}$ and $L$ operators $\psi_{l}:\BT^*\rightarrow\UT$ that are independent of $u$. The condition holds obviously when $\FT$ is a linear operator. However, this is a trivial case, since $D_u\FT(u)$ becomes independent of $u$. Meanwhile,there are also non-trivial cases where this condition also holds. 

\item \textbf{Finite Measurement Nonlinear Operator}:\\
We provide nontrivial cases where the aforementioned finite-dimensionality condition holds. We define a \textbf{finite measurement nonlinear operator} as an operator $\FT$ of the form
\[
[\FT(u), \phi] = F_{\phi}\big(\langle \psi_1(\phi), u \rangle, \ldots, \langle \psi_L(\phi), u \rangle\big), \qquad \forall u \in \UT,
\]
for any $\phi$ in a specific test function set $\BT^*_0$.\\
This class of operators encompasses many important examples, including nonlinear equations in reproducing spaces (to be discussed in Section~\ref{sec:NEinFS}) and neural operators~\cite{lu2021learning,bhattacharya2021model,seidman2022nomad,lanthaler2023parametric}.\\
We will show that, under suitable assumptions, these operators not only admit finite-dimensional solution spaces but are also weakly continuous—a key condition in our convergence analysis.
\item \textbf{A Regularization Framework}: In many data science problems, a regularized loss function composed of fidelity terms and regularization terms is employed to find a robust solution. In the same spirit, for the nonlinear problem with finite measurements \eqref{eq:fmeasure}, we define the following regularization framework:
\begin{equation}\label{eq:reg}
    \min_{u\in\mathcal{U}}\ \ \ \sum_{n=1}^N([\FT(u),\phi_n]-[f,\phi_n])^2+\frac{1}{\mu}\|u\|^2_\mathcal{U}
\end{equation} 
In Theorem \ref{convergeregu}, it is shown that the solution $u^\mu_N$ of the above regularized formulation \eqref{eq:reg} converges to the nonlinear optimal recovery solution $u_N$ as $\mu \rightarrow \infty$. Moreover, $u_N^\mu$ converges to the solution $u^*$  of the nonlinear problem \eqref{eq:nonleq} as $(\mu,N)\rightarrow(\infty,\infty)$.

\item \textbf{A Relaxed Nonlinear Optimal Recovery}: When the measurement test functions $\{\phi_n\}$ do not satisfy the aforementioned finite dimensionality condition, but can be approximated by functions $\{\phi_{nm}\}$ that do: namely, for each fixed $n$, there exist coefficients $\{c_{nm}\}$ such that:
$$\|\sum_{m=1}^M c_{nm}\phi_{nm}-\phi_n\|_{\BT*}\rightarrow 0, \qquad \text{as }M\rightarrow\infty.$$
To obtain a finite-dimensional solution, we can generalize the nonlinear optimal recovery by proposing the following relaxed nonlinear optimal recovery formulation:
\begin{equation}\label{eq:relaxedrec}
\begin{aligned}
\min_{u\in\mathcal{U}}&\ \ \ \ \ \ \ \ \ \|u\|_\mathcal{U}\\
s.t&\ \ \left| [\FT(u),\sum_{m=1}^M c_{mn}\phi_{nm}]-[f,\phi_n]\right|\le \epsilon_{nM},\ n=1,...,N.
\end{aligned}
\end{equation}
In Theorem \ref{convergenceQuasi}, we demonstrate that, under appropriate assumptions, the solution $u^M_N$ of the relaxed nonlinear optimal recovery \eqref{eq:relaxedrec} converges to the solution $u_N$  of the nonlinear optimal recovery problem \eqref{eq:optrec} as $M\rightarrow\infty$, and to the solution $u^*$ of the nonlinear problem \eqref{eq:nonleq} as $(M, N)\rightarrow(\infty,\infty)$.

\item \textbf{Optimal Recovery in Reproducing Function Spaces}: Consider a special case of finite measurement nonlinear operators given by a nonlinear equation of an unknown function $u=u(x)$ on a domain $\Omega$: 
$$\mathcal{F}(u)(x) = F(L_1u(x), \ldots, L_Pu(x),x) = f(x)$$
with the Banach space $\BT$ given by a reproducing kernel Banach space (see definition\ref{RKBS}).
We present 3 key observations on the point evaluations in $\BT^*$ (see definition \ref{pointeva}): 
\begin{itemize}
    \item When the test functions are point evaluations, the nonlinear optimal recovery of the above equation yields finite-dimensional solutions, see {Proposition \ref{OCFunction},\ref{OCQuasiFunction}}.
    \item The span of all point evaluations is dense in the dual space of any reproducing kernel Banach space $\BT$ under the weak* topology, see {Lemma \ref{pointapprox}}.
    \item When the function $F$ is continuous, the operator $\FT$ is weakly continuous in $\BT$.
\end{itemize} 
These observations enable the use of relaxed nonlinear optimal recovery with point evaluations as the approximation function set to obtain a finite-dimensional solution of the nonlinear problem of finite measurements with convergence guarantees.
\item \textbf{Linear-Nonlinear Decomposition}: We consider
a nonlinear equation with a linear-nonlinear decomposition, namely:
$$[\FT(u),\phi]=[Lu,\phi]+[\hat{\FT}(u),\phi]=[f,\phi]$$
where $L$ can be seen as the principle linear part with $\hat{\FT}$ being a nonlinear perturbation that maps into a more regular function space. In this case, we can generalize the relaxed optimal recovery \eqref{eq:relaxedrec} by approximating with point evaluations only those test functions corresponding to the nonlinear term. That is,
\begin{equation}\label{eq:apprelax}
\begin{aligned}
\min_{u\in\mathcal{U}}&\ \ \ \ \ \ \ \ \ \|u\|_{\mathcal{U}}\\
s.t&\ \ \left| [Lu,\phi_n]+[\hat{\FT}(u),\sum_{m=1}^{M} c_{mn}\delta_{x_m}]-[f,\phi_n]\right|\le \epsilon_{nM},\; n=1,...,N.
\end{aligned}
\end{equation}
This generalization significantly reduces the regularity requirement of the solution when using relaxed nonlinear optimal recovery to obtain a finite-dimensional solution. We present the optimality condition and the convergence result of this generalization and a concrete example will be presented in Section \ref{subsec:MOGene}.

\item \textbf{Multi-Domain Generalization}:
Let the domain $\Omega$ have a decomposition, that is, $\Omega=\cup^P_{p=0}\Omega_p$ where $\Omega_0,...,\Omega_P$ are disjoint subdomains. In this case,  we consider a nonlinear equation defined piecewisely on the subdomains corresponding to a nonlinear operator $\FT$ given by:
$$ \FT(u)(x)= \left\{\begin{aligned}
    &Lu(x),\quad&\forall x\in\Omega_p,p=0,\\ &F_p(L_{1,p}u(x),...,L_{Q_p,p}u(x)),\quad&\forall x\in\Omega_p,p=1,...,P
\end{aligned}\right.$$
Applying the point evaluation approximation on subdomains where $
\FT$ is nonlinear, we obtain the following generalized relaxed nonlinear optimal recovery:
\begin{equation}
\label{eq:grelaxed}
\begin{aligned}
\min_{u\in\mathcal{U}}& \ \ \ \|u\|_\mathcal{U}\\
s.t&\ \ \left| [L u,\phi_n]_{\Omega_0}+\sum_{p=1}^P [\FT(u),\sum_{m=1}^{M} c_{mnp}\delta_{x_{mp}}]_{\Omega_p} -[f,\phi_n]\right|\le \epsilon_{nM},\ n=1,...,N.
\end{aligned}
\end{equation}
This generalization is applicable to the solution of nonlinear PDE problems with linear boundary conditions, see the example of a linear natural boundary condition 
represented by a Robin boundary in Section \ref{subsec:MOGene}. As in previous cases, the optimality condition and convergence result for this generalization are presented.
\end{itemize}

To offer a concise summary of the main results,  a diagram is provided in Figure \ref{convergence} to illustrate the relationships among the different problems and the associated solutions $u^*$(\textbf{nonlinear problem}), $u_N$(\textbf{nonlinear optimal recovery}), $u^\mu_N$(\textbf{regularization Problem}) and $u^M_N$(\textbf{
relaxed optimal recovery}).
\begin{figure}[h]
    \centering
    \includegraphics[width=0.75\linewidth]{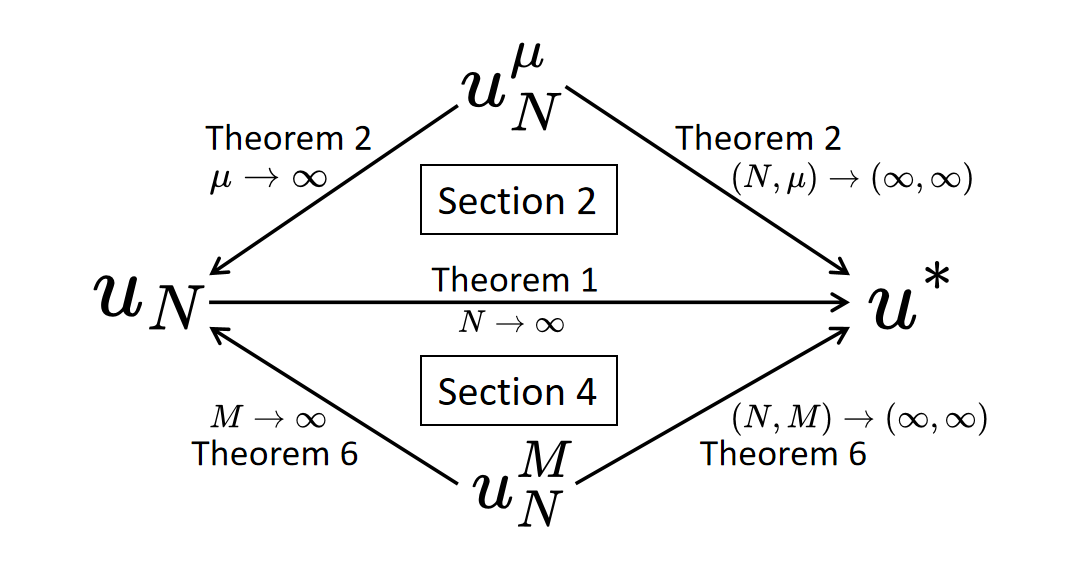}
    \caption{Convergence relation among different solutions}
    \label{convergence}
    \end{figure}
\subsection{Related Works} 
\label{subsec:literature}
There have been many studies on optimal recovery problems. We now discuss those closely related to the results summarized in subsection \ref{subsec:contri}.
\begin{itemize}
    \item \textbf{Optimal Recovery/Minimum Norm Interpolation}: The Optimal Recovery  \cite{micchelli1977survey}, or the Minimum Norm Interpolation, is a widely employed method in regression tasks, such as Ridge regression \cite{mcdonald2009ridge} and Lasso \cite{tibshirani1996regression}. Recently, there has been growing interest in utilizing it to find solutions for partial differential equations (PDEs) \cite{owhadi2019operator,binev2023solving}. However, the extension of optimal recovery to nonlinear problems has also gained attention and has been studied \cite{chen2021solving}.
    \item \textbf{The Representer Theorem and the Reproducing Space}: The representer theorem for infinite-dimensional optimization as a pathway to numerical solvability has been widely studied. The most renowned result is the representer theorem in Reproducing Kernel Hilbert Spaces (RKHS) \cite{10.1007/3-540-44581-1_27}, which characterizes the solution of the optimal recovery problem with the RKHS norm and Dirac measurements as a linear combination of kernels. Additionally, there have also been investigations on the representer theorem in general Hilbert spaces \cite{dinuzzo2012representer}. Moreover, the extension of this theorem to (Reproducing) Banach spaces has been explored \cite{unser2021unifying,JMLR:v22:20-751}.
\end{itemize}

\subsection{Outline of the Paper}
The paper is organized as follows. 

In Section \ref{sec:ORonNP}, we introduce the concept of nonlinear optimal recovery and show its convergence properties. In addition, we present a related regularization problem and demonstrate analogous convergence results.

In Section \ref{sec:OandQNOR}, we present the optimality condition for the solution of the nonlinear optimal recovery problem. Furthermore, we establish a sufficient condition for the solution to be finite-dimensional. Additionally, we propose a relaxed nonlinear optimal recovery algorithm to extend the class of nonlinear equations that admit finite-dimensional solutions. Lastly, we introduce the finite measurement neural operator, a class of operators that satisfy the sufficient condition.

In Section \ref{sec:NEinFS}, we examine a broad class of nonlinear functional equations, a significant subset of nonlinear equations. We demonstrate that point evaluation test functions yield finite-dimensional solutions and the nonlinear operator is weakly continuous. Furthermore, we employ the relaxed nonlinear optimal recovery formulation to extend this result to test functions that can be uniformly approximated by point evaluations.

In Section \ref{sec:MOG}, we generalize the nonlinear optimal recovery and relaxed nonlinear optimal recovery formulations to multi-operator nonlinear equations, accompanied by an analysis of the associated convergence theory. We then explore the realization of these generalizations in functional nonlinear equations, demonstrating how this extension enhances and broadens the problem-solving capabilities of our method.

In Section \ref{sec:ACE}, we present an illustrative example in the form of partial differential equations that demonstrates the efficacy of our algorithms and theoretical framework. Section \ref{sec:CandFD} concludes with additional discussions and outlines potential future research directions. Proofs of all theorems are provided in the Appendix \ref{sec:append}.

\section{Optimal Recovery on Nonlinear Problem}
\label{sec:ORonNP}

In this section, we describe in detail the proposed abstract framework-- \textbf{nonlinear optimal recovery}--to solve general nonlinear problems with finite measurements. In subsection \ref{subsec:ORandRT}, We begin with the introduction to classic optimal recovery and present the representer theorem as well as the extension to linear problems. In subsection \ref{subsec:NOR}, we set up the nonliear optimal recovery method
and explore the theory behind.
We show that as the number of measurements $N\to \infty$, the solutions of the nonlinear optimal recovery 
converge to a solution of the nonlinear problem  \eqref{nonlinearproblem}. Finally in subsection \ref{subsec:RF}, we introduce the regularization problem, a widely used approach in data science. We demonstrate that the regularized solution not only converges to the optimal recovery problem \eqref{NOR}, for a fixed $N$, as the regularization parameter increases but also asymptotically converges to the solution of the nonlinear problem  \eqref{nonlinearproblem}.

\subsection{Optimal Recovery and Representer Theorem}
\label{subsec:ORandRT}
\subsubsection{Classic Optimal Recovery}
Consider the problem of recovering $u$ in a Hilbert space $\UT$ from $N$ measurements
$$\textbf{V}_N(u):=(\langle u,v_1\rangle_\UT
=f_1,...,\langle u,v_N\rangle_\UT
)=f_N\in\mathbb{R}^N $$
where $\langle\cdot,\cdot\rangle$, as used before,  is the inner product in $\UT$, and we use $\|\cdot\|_{\UT}$ to represent the induced norm in subsequent discussions. Meanwhile,  $\{v_n\}_{n=1}^N\subset \UT$ are measurement functions and $\{f_n\}_{n=1}^N$ are measurement data. 

The original \textbf{optimal recovery} \cite{micchelli1977survey,owhadi2019operator} aims to find the best solution operator $\Psi:\mathbb{R}^N\rightarrow \UT$ for the following problem:
$$\min_\Psi\max_{u\in\UT}\ \ \ \ \frac{\|u-\Psi(V_N(u))\|_{\UT}}{\|u\|_{\UT}} .$$

Moreover, given a function $u$, one of the optimal solutions $u_N=\Psi(V_N(u))$ of the problem above can be characterized as the solution to the following \textbf{minimum norm interpolation} problem:
$$\begin{aligned}
\min_{u\in\mathcal{U}}&\ \ \ \ \ \ \ \ \ \|u\|_\mathcal{U}\\
s.t&\ \ \langle u,v_n\rangle_\UT=f_n,\ \ n=1,...,N.
\end{aligned}$$
The \textbf{Representer Theorem} \cite{owhadi2019operator}, which plays an essential role in optimal recovery, shows that the solution $u_N$ of the optimal recovery problem has a finite dimensional representation. Specifically speaking, we have
$$u_N=\sum_{n=1}^N \lambda_n v_n $$
 where $\{\lambda_n\}$ are the coefficients, implying that $u_N$ is the projection of $u$ in the space spanned by the measurement functions in the space $U$.
 
 \subsubsection{Optimal Recovery for Abstract Linear Problems}
The optimal recovery can be easily formulated in an abstract setting for a linear operator
$\LT:\UT\rightarrow\BT^*$.
 Given $N$ measurements $\{[\LT u,\phi_n]\}_{n=1}^N$
 for $N$ measurement function $\{\phi_n\}_{n=1}^N\subset\BT^*$,
 the optimal recovery problem is equivalent to solve 
 $$
\min\,\|u\|_\mathcal{U}\;\;
\text{ among }
u\in\mathcal{U}
\text{ with prescribed } 
\{[\LT u,\phi_n]\}_{n=1}^N\,.
$$

One can prove a similar representer theorem in this case.
Specifically,  
for $1\leq n\leq N$, let $v_n\in \UT$ be the Riesz representer such that 
$$\langle w,v_n\rangle=[\LT w,\phi_n], \; \forall w\in U.$$
That is, $v_n=R_\UT \LT^* \phi_n$ where $\LT^*$ denotes the adjoint operator of $\LT$ and $R_\UT$ is the Riesz mapping in $\UT$ that maps the element in $\UT^*$ to its Riesz representer. Then
the solution $u_N$ satisfies
$$u_N=\sum_{n=1}^N\lambda_nv_n=\sum_{n=1}^N\lambda_n  R_\UT \LT^*\phi_n
$$
with some coefficients $\{\lambda_n\}_{n=1}^N$.

\subsection{Nonlinear Optimal Recovery}
\label{subsec:NOR}
We now focus on the nonlinear equation:
\begin{equation}
\label{nonlinearproblem}
[\FT(u),\phi]=[f,\phi],\ \ \forall\phi\in\BT^*
\end{equation}
where $u\in\UT$, a Hilbert space, and $\FT(u)\in\BT$, a reflexive Banach space. Analogous to the linear optimal recovery approach, with finite measurements $\{[\FT(u),\phi_n]\}$, we introduce the following \textbf{nonlinear optimal recovery} framework:
\begin{equation}
\label{NOR}
\begin{aligned}
\min_{u\in\mathcal{U}}&\ \ \ \ \ \ \ \ \ \|u\|_\mathcal{U}\\
s.t&\ \ [\FT(u),\phi_n]=[f,\phi_n],\ \ n=1,...,N
\end{aligned}
\end{equation}
We now present a general theorem for  \eqref{NOR}, and show that under proper assumptions, the nonlinear optimal recovery problem with $N$ measurements indeed admits a solution. And when $N\rightarrow\infty$, the solution of the optimal recovery problem will converge to the solution of the nonlinear problem  \eqref{nonlinearproblem}. We first present some assumptions.
\begin{assumption}
\textbf{Dense span of test functions}: $\forall \phi\in \mathcal{B}^*$, there is a sequence $\left\{
    \tilde{\phi}_N\in\text{Span}\{\phi_j\}_{j=1}^N \right\}_{N=1}^\infty$ such that $\tilde{\phi}_N$ converge to $\phi$ weak*-ly in $\mathcal{B}^*$.
\label{assu1}
\end{assumption}
\begin{assumption}
\label{assu11} \textbf{Weak continuity}: $u_N$ weakly converging to $u$ in $\mathcal{U}$ implies the weak convergence of $\FT(u_N)$ to $\FT(u)$ in $\BT$.
\end{assumption}
The results are summarized in the following theorem.
\begin{theorem}
\label{convergeNOR} 
Under assumptions \ref{assu1} and \ref{assu11}, 
 the optimal recovery problem \eqref{NOR} for a given $N$ has a solution $u_N$ in $\UT$. Assuming the solution of nonlinear problem \eqref{nonlinearproblem} $u^*$ is also in $\UT$, then $u_N\rightarrow u^*$ strongly in $\mathcal{U}$ as $N\to \infty$.
\end{theorem}
The proof can be found in the Appendix \ref{proofThm1}. 
\subsection{Regularization Framework}
\label{subsec:RF}
Regularization techniques are widely used in data science, where we minimize a loss function composed of a fidelity term and regularization terms. The goal of regularization is to find the best fitting function with some regularity, aligning with the objective of the optimal recovery problem. Similarly, we can define the following \textbf{regularization framework} to find a suitable solution from finite measurements:
\begin{equation}
\label{regularization}
    \min_{u\in\mathcal{U}}\ \ \ \underbrace{\sum_{n=1}^N|[\FT(u),\phi_n]-[f,\phi_n]|^2}_{fidelity}\;\; +\hspace{-.1cm}\underbrace{\frac{1}{\mu}\|u\|^2_\mathcal{U}}_{regularization}
\end{equation}
with a regularization parameter $\mu$. This regularization problem can be interpreted as a relaxation of the nonlinear optimal recovery problem \eqref{NOR}. We can expect that the solution to this regularization problem will converge to the solution of the nonlinear optimal recovery problem as the regularization parameter $\mu$ approaches $0$. We make the following assumption first.
\begin{assumption}
\label{assu2} \textbf{Uniquenss}: The solution $u_N$ of nonlinear optimal recovery problem\eqref{NOR} is unique.

\end{assumption}
In the following theorem, we prove the above convergence. Moreover, we prove that the solution of this regularization problem converges to the solution of the nonlinear problem  \eqref{nonlinearproblem}.
\begin{theorem}
\label{convergeregu}
Under assumptions \ref{assu1}, \ref{assu11} and \ref{assu2}, the regularization problem \eqref{regularization} with $N$ measurements and the regularization parameter $\mu$ has a solution $u^\mu_N$ in $\UT$. In addition, we have,
\begin{itemize}
    \item $u^\mu_N\rightarrow u_N$ strongly in $\UT$ as $\mu\rightarrow \infty$, where 
  $u_N$ is the unique solution of the nonlinear optimal recovery \eqref{NOR}. 
    \item $u^\mu_N\rightarrow u^*$ strongly in $\UT$ as $(\mu,N)\rightarrow(\infty,\infty)$, 
    where $u^*$ is the solution of the nonlinear problem \eqref{nonlinearproblem}.
\end{itemize}
\end{theorem}

A diagram presents the convergence relation among $u^*$(\textbf{Nonlinear Problem} \eqref{nonlinearproblem}), $u_N$(\textbf{nonlinear optimal recovery} \eqref{NOR}), and $u^\mu_N$ (\textbf{regularization problem} \eqref{regularization}) is given in Figure \ref{convergence2}.
\begin{figure}[h]
    \centering
    \includegraphics[width=0.7\linewidth]{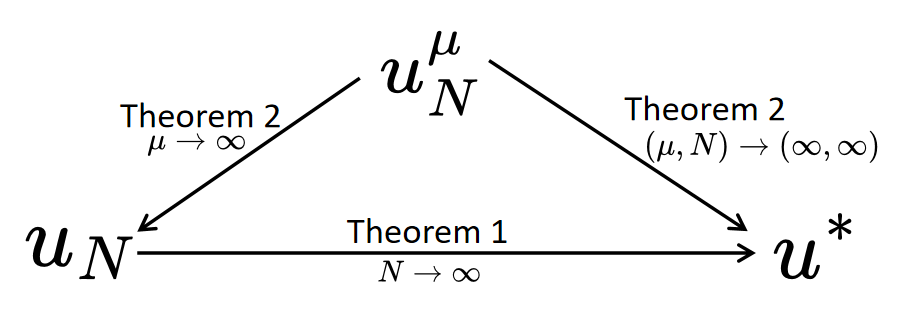}
    \caption{Convergence Relation}
    \label{convergence2}
\end{figure}
\section{Optimality and Relaxation Formulation}
\label{sec:OandQNOR}
In the previous section, we formulated the nonlinear optimal recovery problem for abstract nonlinear equations and demonstrated its desirable convergence properties. This section will delve into the optimality condition and the numerical solvability of the nonlinear optimal recovery  \eqref{NOR}, leading to the proposal of the relaxed nonlinear optimal recovery based on this discussion. In order to consider the finite-dimensional solution

Firstly, subsection \ref{subsec:OC} derives the optimality conditions for the nonlinear optimal recovery problem. Subsequently, subsection \ref{subsec:FDS} will present the condition for the problem to have a finite-dimensional solution. Lastly, subsection \ref{subsec:QNOR} introduces the relaxed nonlinear optimal recovery algorithm, and \ref{subsubsec:CAquasi} will establish its convergence result.
\subsection{Optimality Condition}
\label{subsec:OC}
In classical optimal recovery theory, the representer theorem indicates that the solution will possess a finite-dimensional structure. Analogously, we are interested in investigating the structure of the solution for nonlinear optimal recovery problems \eqref{NOR}. To this end, we will present the following proposition, which characterizes the optimality conditions for the nonlinear optimal recovery problem in any general Hilbert space $\UT$.
\begin{proposition}
\label{OCNonlinear}
     For the Hilbert space $\UT$, by Riesz representation theorem there are $\{\psi(u,\phi)\}\subset\UT$, such that $\forall v,u\in\UT$ and $\forall \phi\in\BT^*$
    $$\langle \psi(u,\phi),v\rangle_\UT=[D_u\FT(u)v,\phi],$$
    where $D_u\FT(u)$ is the Fr\'{e}chet derivative of the mapping $\FT:\UT\rightarrow\BT$ at the point $u$, which is a linear mapping between $\UT$ and $\BT$. Then the solution $u_N$ of the nonlinear optimal recovery problem  \eqref{NOR} satisfies the following condition:
    \begin{equation}
    \label{representerNOR}
        u_N=\sum_{n=1}^N\lambda_n\psi(u_N,\phi_n)
    \end{equation}
    \begin{proof}
        Noticing that we can change the objective function to $\|u\|_\UT^2$ without changing the solution. Then the Lagrangian of the nonlinear optimal recovery problem  \eqref{NOR} is: 
$$\|u\|_\UT^2-\sum_{n=1}^N\lambda_n\left([\FT(u),\phi_n]-[f,\phi_n]\right)$$
Then the Fr\'{e}chet derivative w.r.t $u$ is the following linear operator on $\UT$:
$$[R_\UT^{-1}u,\cdot]_\UT-\sum_{n=1}^N\lambda_n [D_u\FT(u)(\cdot),\phi_n]$$
where $D_u\FT(u)$ is the Fr\'{e}chet derivative of the mapping $\FT:\UT\rightarrow\BT$ at the point $u$, which is a linear mapping between $\UT$ and $\BT$ and $R_\UT$ is the Riesz mapping that maps the element in $\UT^*$ to its Riesz representer in $\UT$. Applying $R_\UT$ we have the optimality condition:
$$u=\sum_{n=1}^N\lambda_n \psi_n(u)$$
where $\psi_n(u)$ is the Riesz representer of $u$
satisfying:
$$\langle \psi_{n}(u),v\rangle_\UT=[D_u\FT(u)v,\phi_n],\ \ n=1,...,N$$
    \end{proof}
    
\end{proposition}
The solution to the regularization framework \eqref{regularization} exhibits an analogous structure, as demonstrated in the following corollary:
\begin{corollary}
The solution $u_N$ of the nonlinear optimal recovery problem  \eqref{regularization} satisfies the following condition:
    \begin{equation}
    \label{representerNOR}
        u_N=\sum_{n=1}^N\lambda_n\psi(u_N,\phi_n)
    \end{equation}
    where $\psi:\UT\times\BT^*\rightarrow \UT$ is
    the Riesz mapping satisfying $\langle \psi(u,\phi),v\rangle_\UT=[D_u\FT(u)v,\phi]$.
    
\end{corollary}
It should be noted that functions $\{\psi(u,\phi_n)\}_{n=1}^N$ generally depend on $u_N$, resulting in the optimality condition  \eqref{representerNOR} being a nonlinear equation in $u_N$. Consequently, this does not constitute a representer theorem in the traditional sense since the solution space is not finite-dimensional.
\subsection{Finite Dimensional Solution}
\label{subsec:FDS}
For nonlinear systems in infinite dimensions, it is in general not numerically solvable. Therefore, we focus on the case where the solution remains finite-dimensional so that the problem becomes numerically solvable.
\subsubsection{Finite Dimensional Condition}
To ensure that the solution in Proposition \ref{OCNonlinear} is spanned by a finite number of functions, the following \textbf{Finite Dimensionality Condition} can be imposed:
\begin{mygraybox}
\textbf{Finite Dimensionality Condition}: Let $\psi:\UT\times\BT^*\rightarrow\UT$ be the Riesz representation mapping satisfies $\langle\psi(u,\phi_n),v\rangle_\UT=[D_u\FT(u)v,\phi_n]$ for $\forall v\in\UT$, there is a constant $L\in\mathbb{N}$ such that:
\begin{equation}
\label{conditionfinite}
    \psi(u,\phi_n)=\sum_{l=1}^{L}c_{l}(u,\phi_n)\psi_{l}(\phi_n),\qquad n=1,...,N
\end{equation}
where $c_{l}:\UT\times\BT^*\rightarrow\mathbb{R}$ are $L$ scalar functions and $\psi_{l}:\BT^*\rightarrow\UT$ are independent of $u$. 
\end{mygraybox}
We will first demonstrate that when this condition is met, our nonlinear optimal recovery problem (Equation \eqref{NOR}) can be transformed into a finite-dimensional optimization problem. According to Proposition \ref{OCNonlinear}, the solution to the nonlinear optimal recovery problem satisfies:
$$ u_N=\sum_{n=1}^N\lambda_n\psi(u_N,\phi_n)=\sum_{n=1}^N\sum_{l=1}^{L}\underbrace{\lambda_nc_{l}(u,\phi_n)}_{Scalar}\psi_{l}(\phi_n)=\sum_{n=1}^N\sum_{l=1}^{L} \lambda_{n,l}\psi_{l}(\phi_n) $$

Therefore minimization of $\|u_N\|_\UT$ with the prescribed finite measurements is equivalent to minimization in
the span $\{\psi_{l}(\phi_n)\}$. Then the nonlinear optimal recovery can be converted to:
$$    
\begin{aligned}
\min_{\Lambda\in\mathbb{R}^{NL}}&\ \ \ \ \ \ \ \ \ \|\sum_{l=1}^{L} \lambda_{n,l}\psi_{l}(\phi_n)\|_\mathcal{U}\\
s.t&\ \ [\FT(\sum_{l=1}^{L} \lambda_{n,l}\psi_{l}(\phi_n)),\phi_n]=[f,\phi_n],\ \ n=1,...,N
\end{aligned}
$$
This results in a finite-dimensional optimization for the unknowns $\{\lambda_{n,l}\}$ that can be numerically solved.

Now we turn to discuss cases where the finite-dimensionality condition will be satisfied. When $\FT$ is a linear operator $\LT$, its Fréchet derivative $D_u\FT(u)$ is equal to $\LT$ for all $u$. Consequently, for any $\phi \in \BT^*$, the functional $\psi(u,\phi)$ satisfies:
$$\langle \psi(u,\phi), v \rangle_\UT = [\LT v, \phi]=[v,\LT^*\phi]=\langle R_\UT \LT^*\phi_n, v \rangle_\UT$$
So we have $\psi(u,\phi_n)=R_\UT \LT^*\phi_n$ and it is independent of $u$. Therefore, any linear operator satisfies the finite-dimensionality condition \eqref{conditionfinite} for $L=1$ and $\psi_1(\phi_n)=R_\UT \LT^*\phi_n$.

We note that requiring the condition \eqref{conditionfinite} to hold for all test functions will significantly limit the class of nonlinear problems to which the framework is applicable. However, there are many more instances where this condition \eqref{conditionfinite} holds for a certain subset of test functions. Namely, for some nonlinear operators $\FT$ there are a subset $\BT^*_0\subset\BT^*$, $L$ scalar functions $c_l:\UT\times\BT^*\rightarrow\mathbb{R}$ and $L$ corresponding functions $\psi_l:\BT^*\rightarrow\UT$ such that for any $\phi\in\BT^*_0$:
\begin{equation}
\label{Quasicondition}
[D_u\FT(u)v,\phi]=\sum_{l=1}^{L}c_l(u,\phi)\langle\psi_l(\phi),v\rangle_\UT,\ \ \forall v\in\UT
\end{equation}
then if our test function set $\{\phi_n\}\subset\BT_0^*$, by proposition \ref{OCNonlinear}, the solution of the nonlinear optimal recovery is finite-dimensional.

\subsection{Relaxed Nonlinear Optimal Recovery}
\label{subsec:QNOR}

As discussed in the previous subsection, the nonlinear operator satisfying \eqref{Quasicondition} possesses a finite-dimensional solution only for a certain test function subset $\BT_0^*$. In practice, it is possible that a test function $\{\phi_n\}$ of the measurements might not belong to $\BT_0^*$. Therefore, the following question arises naturally: Is it possible to modify our nonlinear optimal recovery framework in such a way that we can obtain a finite-dimensional solution without sacrificing the convergence property in this case?    

In this subsection, we give an affirmative answer. Namely, we show that if the span of the subset $\BT_0^*$ is dense in the test function space $\BT^* $, some modifications can be made to the nonlinear optimal recovery problem so that the solution of the new problem retains not only a finite-dimensional representation but also the convergence property. 
\subsubsection{Problem Formulation}
To be more specific, assume that for a fixed measurement $\phi_n$, there exist coefficients $\{c_{nm}\}$ and $\{\phi_{mn}\}\subset\BT_0^*$ such that:
$$\left\|\sum_{m=1}^M c_{nm}\phi_{nm}-\phi_n\right\|_{\BT^*}\rightarrow 0, \qquad \text{as}\ \ M\rightarrow\infty.$$

Then, for each $M\geq 1$, we propose the following
approximation as a formulation of the
\textbf{relaxed nonlinear optimal recovery} of scale-$M$:
\begin{equation}
\label{quasiNOR}
\begin{aligned}
\min_{u\in\mathcal{U}}&\ \ \ \ \ \ \ \ \ \|u\|_\mathcal{U}\\
s.t&\ \ \left| [\FT(u),\sum_{m=1}^M c_{mn}\phi_{mn}]-[f,\phi_n]\right|\le \epsilon_{nM},\ n=1,...,N
\end{aligned}
\end{equation}
where, for each $1\leq n\leq N$, the negative parameters $\{\epsilon_{nM}\geq 0\}$, to be specified later, 
form an infinite sequence that converges to 0 as $M\rightarrow\infty$. Assume that the nonlinear operator $\FT$ satisfies the condition \eqref{Quasicondition}, we can show that the solution $u_N^M$ of the relaxed nonlinear optimal recovery problem is always finite-dimensional:
\begin{proposition}
\label{OCquasiNOR}
Assuming that our nonlinear operator satisfies the condition \eqref{conditionfinite}, then the solution $u^M_N$ of the relaxed nonlinear optimal recovery satisfies:
$$u^M_N\in\text{span}\{\psi_l(\phi_{mn}):l=1,...,L;n=1,...,N;m=1,...,M\}.$$
\end{proposition}

Some discussions are presented below to motivate the use of this inequality formulation:
\begin{itemize}
    \item \textbf{Limitations associated with equality constraints} A key property used earlier to prove the convergence of the nonlinear optimal recovery \eqref{NOR} is that the sequence of solutions can be bounded by the true solution $u$ of the nonlinear problem \eqref{nonlinearproblem}. However, if we simply replace the test function with an approximation in $\BT_0$, the original $u^*$ may no longer satisfy the constraints defining the allowed set of solutions. This leads to a challenging issue of how to guarantee that the solution sequence remains bounded.

\item\textbf{Advantages of inequality constraints} To overcome the above obstacle, we employ a relaxation involving inequality constraints along with a careful choice of the parameter $\epsilon_{nM}$, under reasonable assumptions on the nonlinear term. By incorporating these inequality constraints, we can construct the feasible set in such a way that it includes the solution $u^*$. The detailed derivations and justifications of these conditions will be provided in the next subsection. This ensures that the solution sequence can be bounded by u, which is a crucial requirement to establish convergence. 
\end{itemize}

\subsubsection{Convergence Analysis}
\label{subsubsec:CAquasi}
 Firstly, we show that the weak continuity in the assumption \ref{assu1} of $\FT$ also implies boundness. Namely, we have the following proposition
 \begin{proposition}[\textbf{Boundness of $\FT$}]
\label{wtob}
Under the assumption \ref{assu1}, there exists a constant function $C:\mathbb{R}_+\rightarrow\mathbb{R}_+$ such that for any $K>0$ and $u\in\UT$:
$$\|u\|_\UT\le K\ \ \Rightarrow\ \ \|\FT(u)\|_{\BT}\le C(K)$$
\begin{proof}
    Assuming for a given $K$, there is no such constant $C(K)$. Then there exists a function sequence $\{u_n\}$, such that:
    $$\|u_n\|_\UT\le K,\|\FT(u_n)\|_\BT\ge n,\qquad\forall n=1,...$$
    Since $\{u_n\}$ is bounded, there exists a weakly convergent subsequence $\{u_{n_m}\}$. By assumption \ref{assu1}, $\{\FT(u_{n_m})\}$ is weakly convergent, therefore it is bounded, and that is a contradiction.
\end{proof}
 \end{proposition}
 \begin{assumption}
\label{assu3} 
 (\textbf{Approximation of $\BT^*_0$}): For any $\phi\in \BT^*$, there exist coefficients $\{c_{m}\}$, a function set $\{\phi^*_{m}\}\subset\BT^*_0$,  and coefficients $\{\epsilon_{\phi, M}\}$ satisfies $\epsilon_{\phi, M}\rightarrow 0$ when $M\rightarrow \infty$, such that
$$
\left\|\sum_{m=1}^M c_{m}\phi^*_{m} - \phi\right\|_{\BT^*}\leq \bar{\epsilon}_{\phi, M} .
$$
\end{assumption}

To simplify the notation, in the case of $\phi=\phi_n$, we denote $\bar{\epsilon}_{\phi_n,M}=\bar{\epsilon}_{nM}$.
\begin{remark}
    The above assumption means that $\sum_{m=1}^M c_{nm}\phi_{nm}$ converge to $\phi_n$ strongly in $\BT^*$. In order words,
 $$\overline{\text{Span}\{\phi:\phi\in\BT^*_0 \}}=\BT^* .$$
 In section \ref{subsec:PEandRFS}, we will give a nontrivial example where assumption \ref{assu3} holds.
\end{remark}

 Building upon our earlier discussion, incorporating both the true and the nonlinear optimal recovery solutions into the feasible set confers advantages for the convergence analysis. We show in the following lemma that this inclusion is achievable with the assumption \ref{assu3} and proposition \ref{wtob}:

\begin{lemma}
\label{lemmafeasible}
    Under assumptions \ref{assu3}, by taking $\epsilon_{nM}=C(\|u^*\|)\bar{\epsilon}_{nM}$ with $C$ defined by proposition \ref{wtob}, we have
    \begin{itemize}
        \item The solution for the nonlinear optimal recovery\eqref{NOR} $u_N$ is in the feasible set of relaxed nonlinear optimal recovery \eqref{quasiNOR} of scale $M$.
        \item The solution for the nonlinear equation \eqref{nonlinearproblem} $u^*$ is in the feasible set of relaxed nonlinear optimal recovery \eqref{quasiNOR} of scale $M$.
    \end{itemize}
\begin{proof}
    By previous proofs, we have shown that $\|u_N\|_\UT\le\|u^*\|_\UT$ for any $N>0$, therefore by proposition \ref{wtob}, we can conclude that:
    $$\|\FT(u^*)\|_\BT\le C(\|u^*\|),\ \ \|\FT(u_N)\|_\BT\le C(\|u^*\|).$$
    For any $\phi_n$, we denote the set of correspondent coefficients and functions in assumption \ref{wtob} as $\{c_{nm}\}$ and $\{\phi_{nm}\}$. Then we have:
    $$\begin{aligned}
        &\left|\left[\FT(u_N),\sum_{m=1}^M c_{mn} \phi_{mn}\right]-[\FT(u_N),\phi_n]\right|\\
        \le&\|\FT(u_N)\|_{\BT}\left\|\sum_{m=1}^M c_{mn} \phi_{mn}-\phi_n\right\|_{\BT*}\\
        \le&C(\|u^*\|)\epsilon_{nM}=\delta_{nM}
    \end{aligned}$$
    And the proof for $u^*$ is similar.
\end{proof}
\end{lemma}

By leveraging the previous lemma, which guarantees the boundedness of the solution series, we can establish a theorem regarding the convergence of the solution to the relaxed nonlinear optimal recovery ($u^M_N$) toward the solution of the nonlinear optimal recovery ($u_N$) and the solution of the weak form of the partial differential equation ($u^*$).
\begin{theorem}
\label{ConvergeQNOR}
    Under assumptions \ref{assu1}, \ref{assu11}, \ref{assu2} and  \ref{assu3}, the relaxed nonlinear optimal recovery problem \eqref{quasiNOR} with $N$ measurements and scale-$M$ approximation and constants $\{\delta_{nM}=C(\|u^*\|)\epsilon_{nM}\}$ has a solution $u^M_N$ in $\UT$. In addition, we have,
\begin{itemize}
    \item $u^M_N\rightarrow u_N$ strongly in $\UT$ as $M\rightarrow \infty$, where $u_N$ is the unique solution of the nonlinear optimal recovery \eqref{NOR}. 
    \item $u^M_N\rightarrow u^*$ strongly in $\UT$ as $(M,N)\rightarrow(\infty,\infty)$, 
    where $u^*$ is the solution of the nonlinear problem \eqref{nonlinearproblem}.
\end{itemize}
\end{theorem}
It is important to note that $u^M_N$ converges to $u$ when $(M,N)\rightarrow(\infty,\infty)$, implying that regardless of the specific relationship between $N$ and $M$. That is, as long as both $M$ and $N$ approach infinity, $u^M_N$ will converge to $u$. A diagram representing the various convergence relations among $u^*$(\textbf{nonlinear Equation} \eqref{nonlinearproblem}), $u_N$(\textbf{nonlinear optimal recovery} \eqref{NOR}), and $u^M_N$(\textbf{relaxed nonlinear optimal recovery} \eqref{quasiNOR}) is given in figure \ref{convergenceQuasi}.
\begin{figure}[!h]
    \centering
    \includegraphics[width=0.7\linewidth]{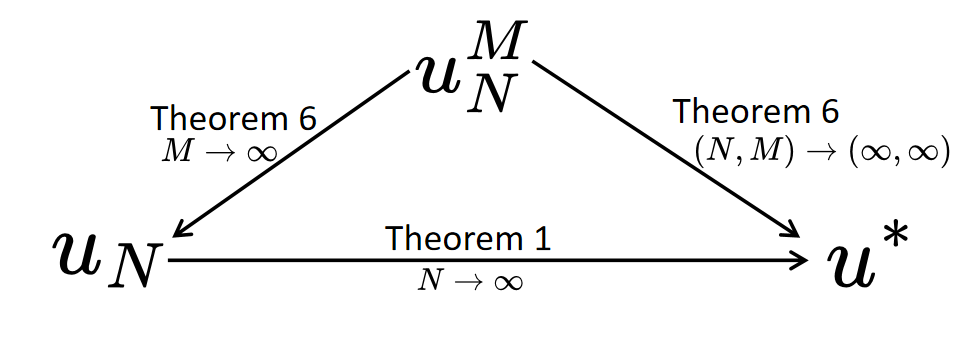}
    \caption{Convergence relation}
    \label{convergenceQuasi}
\end{figure}

\subsection{Finite Measurement Nonlinear Operator}
\label{subsec:FMNO}
The condition(\ref{Quasicondition}) might not appear as easily verifiable. In this subsection, we will define a class of nonlinear operators called \textbf{finite measurement nonlinear operator}, which satisfies the aforementioned condition(\ref{Quasicondition}) and encompasses a wide range of application problems.

We define a nonlinear operator as a \textbf{finite measurement nonlinear operator} if, for any $\phi \in \BT^*_0$, the operator $\FT$ satisfies
\[
[\FT(u),\phi] = F_{\phi}\big(\langle \psi_1(\phi), u \rangle, \ldots, \langle \psi_L(\phi), u \rangle\big), \qquad \forall u \in \UT,
\]
where each $\psi_l(\phi) \in \UT^*$ depends on $\phi$, and $F_\phi$ is a function mapping $\mathbb{R}^L$ to $\mathbb{R}$. 
\subsubsection{Some Examples}
Some examples of finite measurement nonlinear operators are listed below:

\begin{itemize}
    \item \textbf{Function Space Setting}:\\
    Suppose $\UT$ and $\BT$ are both function spaces over a domain $\Omega$, and the test function set $\BT^*_0$ consists of point evaluation functionals on $\Omega$. Assume each $\psi_l(\phi)$ is a linear operator composed with a point evaluation. Then the nonlinear operator $\FT$ takes the form
    \[
    \FT(u)(x) = F\big(L_1 u(x), \ldots, L_Q u(x),x\big), \qquad \forall x \in \Omega,\ \forall u \in \UT,
    \]
    where $L_1, \ldots, L_Q$ are linear operators.\\
    This setting covers a wide range of problems, including nonlinear differential equations, integral equations, and operator equations. We will discuss this case in detail in Section~\ref{sec:NEinFS}.
    \item \textbf{Neural Operator Setting}:\\
    Suppose $\BT$ is a function space over a domain $\Omega$, and the test function set $\BT^*_0$ consists of appropriate evaluations. Assume that $F_x$ is a neural network function. In previous work~\cite{lanthaler2023parametric} on operator learning, a neural network-type operator is defined as
    \[
    \FT(u)(x) = F_x\big(\langle u, \psi_1 \rangle, \ldots, \langle u, \psi_L \rangle\big), \qquad \forall x \in \Omega,\ \forall u \in \UT,
    \]
    where $\psi_1, \ldots, \psi_L$ are test functions in $\BT^*$, and $F_x$ is a neural network function.\\
    This definition includes most operator learning frameworks, such as PCANet~\cite{bhattacharya2021model}, DeepONet~\cite{lu2021learning}, and NOMAD~\cite{seidman2022nomad}.
\end{itemize}
\subsubsection{Finite Dimensional Solution}
In this subsection, we show that any finite measurement nonlinear operator satisfies condition~\eqref{Quasicondition}. We summarize this result in the following lemma:
\begin{lemma}
\label{lemma:finitedim}
Let $\FT$ be a finite measurement nonlinear operator. Then $\FT$ satisfies condition~\eqref{Quasicondition}.
\end{lemma}

\begin{proof}
    For any $\phi\in\BT^*_0$, by definition we have:
    \[
    [\FT(u),\phi] = F_{\phi}\big(\langle \psi_1(\phi), u \rangle, \ldots, \langle \psi_L(\phi), u \rangle\big), \qquad \forall u \in \UT,
    \]
    Computing the Frech\'{e}t derivative w.r.t to $u$:
    \[
    [D_u\FT(u)v,\phi] = \sum_{l=1}^L \partial_lF_{\phi}\big(\langle \psi_1(\phi), u \rangle, \ldots, \langle \psi_L(\phi), u \rangle\big)\langle\psi_l(\phi), v\rangle, \qquad \forall u,v \in \UT,
    \]
    Noticing that $\partial_lF_{\phi}$ maps $\mathbb{R}^L$ to $\mathbb{R}$, this is exactly condition\eqref{Quasicondition}.
\end{proof}
\subsubsection{Convergence Analysis--Weak Continuity}
In Theorems~\ref{convergeNOR}, \ref{convergeregu}, and \ref{ConvergeQNOR}, one of the crucial assumptions is Assumption~\ref{assu11}, which requires the weak continuity of the nonlinear operator. However, weak continuity is generally difficult to verify. In this subsection, we show that for finite measurement nonlinear operators, this condition can be reduced to the continuity of $F_\phi$ and Assumption~\ref{assu3}. Specifically, we have the following assumption and lemma:
\begin{assumption}
\label{assu6}
    For the finite measurement nonlinear operator $\FT$, the function $F_\phi:\mathbb{R}^L\rightarrow\mathbb{R}$ is continuous for every $\phi\in\BT^*_0$.
\end{assumption}
\begin{lemma}
\label{lemma:wc}
    Let $\FT$ be a finite measurement nonlinear operator satisfying Assumption~\ref{assu6}, and suppose the test function set $\BT^*_0$ satisfies Assumption~\ref{assu3}. Then $\FT$ satisfies Assumption~\ref{assu11}; that is, $\FT$ is weakly continuous.
\end{lemma}
\begin{proof}
    Since $\BT^*_0$ is dense in $\BT^*$, for any $\phi\in\BT^*$ there are a sequence $$\bar{\phi}_K=\sum_{k=1}^Kc_k\phi_k\rightarrow\phi,\ \ \ \ K\rightarrow\infty$$
    where $\{\phi_k\}\subset\BT^*_0$.

    By triangle inequality, we have:
    $$|[\FT(u_n),\phi]-[\FT(u),\phi]|\le|[\FT(u_n),\phi]-[\FT(u_n),\tilde{\phi}_K]|+|[\FT(u_n),\tilde{\phi}_K]-[\FT(u),\tilde{\phi}_K]|+|[\FT(u),\phi]-[\FT(u),\tilde{\phi}_K]|$$
    and we want to show that for any $\epsilon>0$, there is a $n,k$ such that the RHS quantity is less than $\epsilon$.
    
    Since $\{u_n\}$ is a weakly convergence sequence, therefore it is bounded, also by $\FT$ is bounded, $\{\FT(u_n)\}$ is a bounded sequence. Then there exist a $K$ such that $$ |[\FT(u_n),\phi]-[\FT(u_n),\tilde{\phi}_K]|\le\frac{\epsilon}{3},\quad |[\FT(u),\phi]-[\FT(u),\tilde{\phi}_K]|\le\frac{\epsilon}{3}$$
    And for each $\phi_k$ with $k=1,...,K$, we have $$[\FT(u_n),\phi_k]=F_{\phi_k}([u_n,\psi_1],...,[u_n,\psi_L])\rightarrow F_{\phi_k}([u,\psi_1],...,[u,\psi_L])=[\FT(u),\phi_k] $$
    by the weak convergence of $u_n$ and the continuity of $F_{\phi_k}$. Therefore there exists a $n$ such that:
    $$|[\FT(u_n),\tilde{\phi}_K]-[\FT(u),\tilde{\phi}_K]|\le\frac{\epsilon}{3}$$
    Combining all these, we proved the result.
\end{proof}
\section{Nonlinear Optimal Recovery in Reproducing Function Spaces}
\label{sec:NEinFS}

In this section, we focus on the \textbf{function space setting} of \textbf{finite measurement nonlinear operators}. This encompasses a wide range of application problems. Specifically, we will study the \textbf{nonlinear equation in reproducing function spaces}, which can represent many widely known examples.

In Subsection \ref{subsec:NEinFS}, we first
introduce the nonlinear equation problem in a general function space related to functions defined on an open domain $\Omega$. In Subsection\ref{subsec:PEandRFS}, we define the point evaluation and reproducing function spaces and present two observations related to the condition \eqref{conditionfinite} and assumption \ref{assu3}. Moreover, in Subsection \ref{subsec:FDC} we present some cases where our nonlinear optimal recovery  \eqref{NOR} and relaxed nonlinear optimal recovery  \eqref{quasiNFOR} admit a finite-dimensional solution.

\subsection{Nonlinear Equation in function spaces}
\label{subsec:NEinFS}
Given a domain $\Omega$, for an unknown function $u:\Omega  \to \mathbb{R}$, we 
consider the following general nonlinear equation:
$$
    \FT(u)(x)=f(x),\ \ \ \ x\in\Omega
$$
where $\FT$ is a nonlinear operator on the space of functions on $\Omega$ and $f$ is a prescribed function.

\paragraph{Some examples} Depending on the choices of $\Omega$, the nonlinear equation may represent many different types of model problems. We offer several scenarios to illustrate the broad applicability of our proposed setting.

\begin{itemize}
    \item \textbf{$\Omega$ is a countable set}. In this setup, our equation can be an infinite system of equations.
    \item \textbf{$\Omega$ is a subspace of $\mathbb{R}^d$}. If $\FT$ is a nonlinear differential or integral operator, 
   then our equation can represent a nonlinear differential or integral equation.
    \item \textbf{$\Omega$ is a function space}. Then the nonlinear equation can represent an operator equation.
\end{itemize}

Moreover, we assume the general form of our nonlinear operator $\FT$ to be a finite measurement nonlinear operator on the whole domain $\Omega$ with the following form:
\begin{equation}
\label{generaloperator}
    \FT(u)(x)=F(L_1u(x),...,L_Qu(x),x),\ \ \ x\in\Omega 
\end{equation}
where $L,L_1,...,L_Q$ represent various linear operators that map $\UT$ to some function spaces on the domain $\Omega$ and $F$ is a nonlinear function from $\mathbb{R}^Q$ to $\mathbb{R}$.
\subsection{Point Evaluation and Reproducing Function Spaces}
\label{subsec:PEandRFS}
For a function space $\BT$ defined on $\Omega$, let us recall the notion of a point evaluation.

\begin{definition}
\label{pointeva}
The point evaluation functional $\delta_x$ associated with a point $x\in\Omega$ is a linear functional on $\BT$, such that 
$$[f,\delta_x]=f(x),\quad\forall\,
f\in\BT.$$
\end{definition}

Note that $\delta_x$ often can also be viewed as a Dirac delta measure.

There have been extensive studies of function spaces where any point evaluation is a bounded functional. Such function spaces are called the \textbf{Reproducing Function Space}. Moreover, if the function space is a Hilbert(Banach) space, then it is a \textbf{Reproducing kernel Hilbert (Banach) space} \cite{steinwart2008support}\cite{zhang2009reproducing}. Specifically, we have the following definition:
\begin{definition}
\label{RKBS}
    A Banach space $\BT$ on domain $\Omega$ is a reproducing kernel Banach space if for any $x\in\Omega$, the point evaluation functional $\delta_x$ is bounded. Namely, there exists a constant $L_x$, such that for $\forall f\in\BT$,
    $$f(x)=[f,\delta_x]\le L_x\|f\|_{\BT}$$
\end{definition}

First, we have the following observation 
\begin{mygraybox}
    \textbf{Observation 1}: When the test functions are point evaluations, any nonlinear function equation with a nonlinear operator of the form (\ref{generaloperator}) satisfies the finite-dimensional condition (\ref{conditionfinite}).
\end{mygraybox}
This observation follows directly from the fact that $\FT$ is a finite measurement neural operator, according to lemma~\ref{lemma:finitedim} and:
$$\FT(u)(x)=[\FT(u),\delta_x]=F_x(\langle u,L_1^*\delta_x\rangle,...,\langle u,L_Q^*\delta_x\rangle)=F(L_1u(x),...,L_Qu(x),x),\ \ \ x\in\Omega $$
where we denote $F_x(\cdot) = F(\cdot,x)$.

Moreover, we note that the convergence analysis of the relaxed nonlinear optimal recovery is subject to a crucial assumption on the approximation function set $\BT^*_0$. Specifically, Assumption \ref{assu3} posits that the span of the approximation function set can approximate any function $\phi \in \BT^*$. We are then ready to present the next important observation:
\begin{mygraybox}
    \textbf{Observation 2}: For any \textbf{reflexive} reproducing kernel Banach space, the set of all point evaluations satisfies the assumption \ref{assu3}.
\end{mygraybox}

To formalize this, we present the following lemma:
\begin{lemma}
\label{pointapprox}
    For any reproducing kernel Banach space $\BT$ on domain $\Omega$, the assumption \ref{assu3} is true for the subset $\BT^*_0=\{\delta_x:x\in\Omega\}$. Namely,
    $$\overline{\text{Span}\{\delta_x:x\in\Omega\}}=\BT^*$$
\begin{proof}
    By the proposition 2 in \cite{wang2024sparse}, we have:
    $$\overline{\text{Span}\{\delta_x:x\in\Omega\}}^{w^*}=\BT^*.$$
    Since $\BT^*$ is reflexive, the weak topology is the same as the weak* topology. Also, by Theorem 3.7 in \cite{brezis2011functional}, any convex set is weakly closed if and only if it is strongly closed, therefore:
    $$\overline{\text{Span}\{\delta_x:x\in\Omega\}}=\overline{\text{Span}\{\delta_x:x\in\Omega\}}^{w^*}=\BT^*.$$
\end{proof}
\end{lemma}
\begin{remark}
    Noticing that by observations 1 and 2, only assumptions \ref{assu1}, \ref{assu11} and \ref{assu2} are needed for the convergence of relaxed nonlinear optimal recovery \eqref{quasiNOR}, which is the same set of assumptions as for the case of the nonlinear optimal recovery \eqref{NOR}. In other words, when $B^*$ is a reproducing kernel Banach space, we can always employ the relaxed nonlinear optimal recovery with approximations by pointwise evaluations, without imposing additional assumptions.
\end{remark}
Lastly, the discussion of weak continuity is presented in the next observation:
\begin{mygraybox}
    \textbf{Observation 3}: Assume $\BT$ is a reflexive reproducing Banach space and assumption~\ref{assu6}($F$ is continuous), $\FT$ satisfies assumption~\ref{assu11}($\FT$ is weak continuous).
\end{mygraybox}
Noticing that by observation 2, the span of the test function set $\BT^*_0=\{\delta_x;x\in\Omega\}$ is dense in $\BT^*$. Then this observation follows directly from the Lemma~\ref{lemma:wc}

\subsection{Finite Dimensional Solutions}
\label{subsec:FDC}
Based on the discussion in the previous subsection, we present two cases where we can obtain a finite-dimensional solution for nonlinear optimal recovery on a reflexive reproducing kernel Banach space:
\begin{itemize}
    \item \textbf{Nonlinear Optimal Recovery}: When test functions are point evaluations.
    \item \textbf{Relaxed Nonlinear Optimal Recovery}: More general test functions.
\end{itemize}
In this subsection, we let $\UT$ be a reproducing Hilbert function space, $\BT$ and $\BT^*$ be a duality pair of reproducing Banach function spaces, all on the domain $\Omega$.
Meanwhile, let $\FT:\UT\rightarrow\BT$ be a nonlinear operator and take the form \eqref{generaloperator}. 
\subsubsection{The case of Nonlinear Optimal Recovery}
Now if the $N$ test functions are points evaluation $\{\delta_{x_n}\}$ on  the set of points $\{x_n\}$, the nonlinear optimal recovery \eqref{NOR} on functions spaces becomes
    \begin{equation}
\label{NFORPoint}
\begin{aligned}
\min_{u\in\mathcal{U}}&\ \ \ \ \ \ \ \ \ \|u\|_\mathcal{U}\\
s.t&\ \ \FT(u)(x_n)=f(x_n),\ \ n=1,...,N
\end{aligned}
\end{equation}
Combining Lemma \ref{lemma:finitedim} and Proposition \ref{OCNonlinear}, we arrive at the following proposition:
\begin{proposition}
\label{OCFunction}
     The solution $u_N$ of the nonlinear optimal recovery problem with point evaluation measurements  \eqref{NFORPoint} satisfies the following condition:
    \begin{equation}
    \label{representerNFOR}
        u_N=\sum_{n=1}^N\sum_{q=1}^Q\lambda_{n,q}\psi_{n,q}
    \end{equation}
    where $\psi_{n,q}$ are given by Riesz representation theorem, satisfies $[L_qv,\delta_{x_n}]=\langle\psi_{n,q},v\rangle_\UT$
\end{proposition}
\begin{remark}
    This formulation aligns with that presented by Chen et al. \cite{chen2021solving}, who established the convergence theory and solution structure for this case. Our theoretical framework offers an alternative approach to deriving the results reported in their study.
\end{remark}
\subsubsection{The Case of Relaxed Nonlinear Optimal Recovery}
By Lemma \ref{pointapprox}, any function $\phi\in\BT^*$ can be approximated by some point evaluations. Therefore consider the scale-$M$ approximation of the test function $\phi_n$:
$$\left\|\sum_{m=1}^Mc_{mn}\delta_{x_{mn}}-\phi_n\right\|_{\BT*}\rightarrow 0 $$
Then the scale-$M$ relaxed nonlinear optimal recovery becomes
\begin{equation}
\label{quasiNFOR}
\begin{aligned}
\min_{u\in\mathcal{U}}&\ \ \ \ \ \ \ \ \ \|u\|_\mathcal{U}\\
s.t&\ \ \left| \sum_{m=1}^M c_{mn}\FT(u)(x_{mn})-[f,\phi_n]\right|\le \epsilon_{nM},\ n=1,...,N
\end{aligned}
\end{equation}
Combining Lemma \ref{lemma:finitedim} and Proposition \ref{OCquasiNOR} we have:

\begin{proposition}
\label{OCQuasiFunction}
The solution $u_N$ of the relaxed nonlinear optimal recovery problem with point evaluation measurements  \eqref{quasiNFOR} satisfies the following condition:
\begin{equation}
\label{representerNFOR}
u^M_{N}=\sum_{n=1}^N\sum_{m=1}^M\sum_{q=1}^Q\lambda_{n,m,q}\psi_{n,m,q}
    \end{equation}
    where $\psi_{m,q}$ are given by the Riesz representation theorem, satisfies $[L_qv,\delta_{x_{mn}}]=\langle\psi_{n,m,q},v\rangle_\UT$
\end{proposition}
\begin{remark}
    This formulation extends the model proposed by Chen et al. \cite{chen2021solving}, expanding its applicability to more general problems by accommodating cases where measurements are not limited to precise point evaluations.
\end{remark}
\section{Some Generalizations in Function Spaces}
\label{sec:MOG}
This section presents two generalizations of our relaxed nonlinear optimal recovery framework in function spaces. Subsection \ref{subsec:OD} examines the separate consideration of linear and nonlinear components, exploring the potential for reducing regularity requirements through this decomposition. Subsection \ref{subsec:MDG} introduces a generalization where the nonlinear operator assumes distinct forms across different subdomains. In each subsection, we will demonstrate the convergence of the proposed generalization.
\subsection{Linear-Nonlinear Decomposition}
\label{subsec:OD}
Section \ref{sec:NEinFS} demonstrates that employing point evaluations as test functions ensures the satisfaction of the finite-dimensional condition \eqref{conditionfinite}. Nevertheless, this approach demands $\mathcal{B}$ to be a reproducing kernel Banach space, a requirement that entails exceptionally high regularity within the framework of traditional PDE theory. In other words, to make the
formulation presented earlier \eqref{quasiNFOR}
work with the solution $u^*$, the function \textbf{$\mathcal{F}(u^*)$ must be pointwise defined and bounded}.
On the other hand, for many PDEs, it is possible to decompose the nonlinear operator into two parts: a linear component that retains the primary regularity requirements, and a residual nonlinear component. Such a decomposition may be utilized to lower the regularity assumptions of the optimal recovery formulation on the solution. 

In such cases, we define the function space for the linear term $L(u)$ as $\BT_L$ and for the nonlinear term $\FT(u)$ as $\BT_N$. Then the nonlinear function equation \eqref{nonlinearproblem} on domain $\Omega$ can be reformulated as:
$$[\FT(u),\phi]=[L u,\phi]+[\hat{\FT}(u),\phi]=[f,\phi]$$
for any test function $\phi\in\BT_L^*\cap\BT_N^*=\BT^*$. Same as \eqref{generaloperator}, we assume the $\hat{\FT}$ to be a finite measurement nonlinear operator with the following general form:
$$\hat{\FT}(u)(x)=\hat{F}(L_1u(x),...,L_Qu(x),x)$$
where $L,L_1,...,L_Q$ are different linear operators maps $\UT$ to some function spaces with domain $\Omega$ and $\hat{F}$ is a nonlinear function from $\mathbb{R}^Q$ to $\mathbb{R}$.
\subsubsection{Formulation and Optimality Condition}
For the linear part $L$, we retain the test function $\phi_n$, while employing the point evaluations approximation $\sum_{m=1}^{M} c_{mn}\delta_{x_m}$ of $\phi_n$ for the nonlinear component. This approach allows us to define the corresponding multi-operator relaxed nonlinear optimal recovery as follows:
\begin{equation}
\label{Linear-Nonlinear Case}
\begin{aligned}
\min_{u\in\mathcal{U}}&\ \ \ \ \ \ \ \ \ \|u\|_\mathcal{U}\\
s.t&\ \ \left| [Lu,\phi_n]+[\hat{\FT}(u),\sum_{m=1}^{M} c_{mn}\delta_{x_{mn}}]-[f,\phi_n]\right|\le \epsilon_{nM},\ n=1,...,N
\end{aligned}
\end{equation}
It is important to note that this formulation only requires \textbf{the nonlinear component $\hat{\FT}(u)$ to be pointwise defined}, while the primary regularity requirements are imposed on the linear part. Consequently, this modification allows for the obtainment of solutions with lower regularity compared to the original approach. A concrete example of this will be presented in the section \ref{sec:ACE}.

Then the solution $u_N$ of  \eqref{Linear-Nonlinear Case} satisfies:
$$ u^M_N=\sum_{n=1}^N\lambda_n\psi_n+\sum_{n=1}^N\sum_{m=1}^M\sum_{q=1}^Q\lambda_{n,m,q}\psi_{n,m,q} $$
where $\psi_n,\psi_{n,q}$ are given by Riesz representation theorem, satisfies $ [Lv,\phi_n]=\langle v,\psi_n\rangle_\UT $ and $[L_qv,\delta_{x_{mn}}]=\langle v,\psi_{n,m,q}\rangle_\UT$.
\subsubsection{Convergence Analysis}
Note that in the formula above, we only apply the test function approximation on the nonlinear part. Therefore, one should expect a convergence theorem similar to the relaxed nonlinear optimal recovery. Accordingly, we make a similar weak convergence assumption and propose the following convergence result:
\begin{assumption}
\label{assu4}
$u_N$ weakly converging to $u$ in $\mathcal{U}$ implies $\LT u_N$ weakly converging to $L u$ in $\BT_L$.
\end{assumption}
\begin{corollary}
     Under assumptions \ref{assu1}, \ref{assu2}, \ref{assu4}, $\BT_N$ is a reflexive reproducing Banach space and $\hat{F}$ is continuous, the relaxed nonlinear optimal recovery problem for the linear-nonlinear decomposition \eqref{Linear-Nonlinear Case} with $N$ measurements and scale-$M$ approximation has a solution $u^M_N$ in $\UT$. In addition,  by taking $\epsilon_{nM}=C(\|u^*\|)\hat{\epsilon}_{nM}$,
\begin{itemize}
    \item $u^M_N\rightarrow u_N$ strongly in $\UT$ as $M\rightarrow \infty$, where $u_N$ is the unique solution of the nonlinear optimal recovery \eqref{NOR}. 
    \item $u^M_N\rightarrow u^*$ strongly in $\UT$ as $(M,N)\rightarrow(\infty,\infty)$, 
    where $u^*$ is the solution of the nonlinear problem \eqref{nonlinearproblem}.
\end{itemize}
\end{corollary}
\begin{remark}
    Since our considerations are in the reproducing kernel Banach space and the approximation is made by point evaluations, assumption \ref{assu3} is automatically satisfied by Lemma \ref{pointapprox}.
\end{remark} 
\subsection{Multi-Domains Generalization}
\label{subsec:MDG}
In our previous discussion of nonlinear equations on the domain $\Omega$, we assumed the nonlinear operator $\FT$ had the uniform form \eqref{generaloperator} throughout the entire domain. However, the nonlinear operator may sometimes be defined piecewise on different subdomains. Consider a decomposition of the domain $\Omega=\cup^P_{p=1}\Omega_p$, where $\Omega_1,...,\Omega_P$ are disjoint subdomains. In this case, we assume $\FT$ takes the following form:
$$ \FT(u)(x)= F_p(L_{1,p}u(x),...,L_{Q_p,p}u(x),x),\qquad \forall x\in\Omega_p,p=1,...,P$$ where $L_{q,p}$ are different linear operators maps $\UT$ to some function spaces with domain $\Omega$ and $F_p$ are nonlinear function from $\mathbb{R}^{Q_p}$ to $\mathbb{R}$.Then we can rewrite the nonlinear equation as:
$$   [\FT(u),\phi]_{\Omega_1}+...+[\FT(u),\phi]_{\Omega_P}=[f,\phi]_\Omega$$
where $[\cdot,\cdot]_{\Omega_p}$ is the duality pairing on the subdomain $\Omega_p$.

There are numerous scenarios where this generalization proves valuable. For instance, it is particularly well-suited to piecewise equations. Furthermore, boundary conditions play a crucial role in many PDEs, and in the case of natural boundary conditions, the boundary values become an integral part of the variational equation. Consequently, these situations align perfectly with this generalization. A concrete example of this will be presented in the section \ref{sec:ACE}.
\subsubsection{Formulation and Optimality Condition}
Regarding the relaxed nonlinear optimal recovery, we consider the point evaluation approximation  $\sum_{m=1}^{M} c_{mnp}\delta_{x_{mnp}}$ of $\phi_n$ on subdomains $\{\Omega_p\}_{p=1}^P$ where $\FT$ is nonlinear,
and arrive at the following relaxed nonlinear optimal recovery formulation:
\begin{equation}
\label{quasiNORM}
\begin{aligned}
\min_{u\in\mathcal{U}}& \ \ \ \|u\|_\mathcal{U}\\
s.t&\ \ \left| [\FT(u),\sum_{m=1}^{M} c_{mn1}\delta_{x_{mn1}}]_{\Omega_1}+\cdots+[\FT(u),\sum_{m=1}^{M} c_{mnP}\delta_{x_{mnP}}]_{\Omega_P}-[f,\phi_n]\right|\le \epsilon_{nM},\ n=1,...,N
\end{aligned}
\end{equation}

Similar to Proposition \ref{OCquasiNOR}, the relaxed nonlinear optimal recovery solution $u^M_N$ satisfies:
$$ u^M_N=\sum_{n=1}^N\sum_{p=1}^P\sum_{m=1}^M\sum_{q=1}^Q\lambda_{n,m,q,p}\psi_{n,m,q,p},$$
where $\psi_{n,m,q,p}$ are given by Riesz representation theorem $[L_qv,\delta_{x_{mn}}]_{\Omega_p}=\langle v, \psi_{n,m,q,p}\rangle_\UT$.
\subsubsection{Convergence Analysis}
Similar to Lemma \ref{pointapprox}, we can show that the point evaluation approximation works on every subdomain $\Omega_p$ for $p=1,...P$. Our convergence result is:
\begin{corollary}
     Under assumptions \ref{assu1}, \ref{assu2}, $\BT$ is a reflexive reproducing Banach space and $F_p$ is continuous, the relaxed nonlinear optimal recovery problem for the multi-domains generalization \eqref{quasiNORM} with $N$ measurements and scale-$M$ approximation has a solution $u^M_N$ in $\UT$. In addition, by taking $\delta_{nM}=C(\|u^*\|)\epsilon_{nM}$,
\begin{itemize}
    \item $u^M_N\rightarrow u_N$ strongly in $\UT$ as $M\rightarrow \infty$, where $u_N$ is the unique solution of the nonlinear optimal recovery \eqref{NOR}. 
    \item $u^M_N\rightarrow u^*$ strongly in $\UT$ as $(M,N)\rightarrow(\infty,\infty)$, 
    where $u^*$ is the solution of the nonlinear problem \eqref{nonlinearproblem}.
\end{itemize}
\end{corollary}
\section{A Concrete Example}
\label{sec:ACE}
In this section, we will work with a concrete example--a nonlinear PDE for the steady state of a reaction-diffusion system on a domain $\Omega\subset\mathbb{R}^3$
\begin{equation}
\label{PDEexample}
-\Delta u(x)+u^3(x)=f(x),\ \ x\in\Omega
\end{equation}
For the homogeneous Dirichlet boundary condition, Subsection \ref{subsec:NORandQNOR} demonstrates the corresponding nonlinear optimal recovery and relaxed nonlinear optimal recovery formulations for this PDE example. We explore the optimal regularity that satisfies all requirements. Furthermore, Subsection \ref{subsec:MOGene} illustrates how the generalization introduced in Section \ref{sec:MOG} helps us to improve or generalize our previous formulation.
\subsection{Nonlinear Optimal Recovery and relaxed nonlinear Optimal Recovery}
\label{subsec:NORandQNOR}
For the homogeneous Dirichlet boundary condition,  the equation \eqref{PDEexample} becomes:
\begin{equation}
\label{PDEexampleD}
    \left\{\begin{aligned}
    -\Delta u(x)+u^3(x)=f(x),\ \ x\in\Omega \\
    u(x)=0\ \ ,\ \ x\in\partial\Omega
\end{aligned}\right.
\end{equation}
and by integration by part, the weak form for this PDE is defined as:
$$\int_\Omega \nabla u(x)\cdot\nabla \phi(x)+u^3(x)\phi(x)dx=\int_\Omega f(x)\phi(x)dx,\ \ \ \ \forall\phi\in H^1_0(\Omega).$$
According to Sobolev embedding theory, in $\mathbb{R}^3$, 
$H^1_0(\Omega)$
embeds into $L^4(\Omega)$ embeds continuously, therefore in this case, the standard test function space $\BT^*=H^1_0(\Omega)\cap L^4(\Omega)=H^1_0(\Omega)$ and $\BT = H^{-1}(\Omega)$. By the classic PDE theory, the minimal requirement on our solution space is that $\UT$ is a Hilbert space \textbf{continuously embedded} in $H^1_0(\Omega)$. 

\subsubsection{Nonlinear Optimal Recovery}
The nonlinear optimal recovery \eqref{NOR} of the equation \eqref{PDEexample} and test function $\{\phi_n\}$ is:
    \begin{equation}
\label{NFORPDE}
\begin{aligned}
\min_{u\in\mathcal{U}}&\ \ \ \ \ \ \ \ \ \|u\|_\mathcal{U}\\
s.t&\ \ \int_\Omega \nabla u(x)\cdot\nabla \phi_n(x)+u^3(x)\phi_n(x)dx=\int_\Omega f(x)\phi_n(x)dx,\ \ n=1,...,N
\end{aligned}
\end{equation}

We will show that to satisfy the assumptions in Theorem \ref{convergeNOR}, the minimal regularity requirement is still $\mathcal{U}$ to be continuously embedded in $H^1_0(\Omega)$.
\begin{itemize}
        \item \textbf{Dense span of test functions}: We can adopt standard practices to pick a series of test functions $\{\phi_n\}$ such that its span is dense in $H^1_0(\Omega)$.
    \item \textbf{Weak continuity for linear part}: For the linear part $-\Delta u$, we need $\nabla u_N$ convergent to $\nabla u$ weakly in $L^2(\Omega)$ if $u_N$ converges to $u$ weakly in $\UT$, which is direct consequence of $\UT$ is \textbf{continuously embedded} into $H^1_0(\Omega)$.

    \item \textbf{Weak continuity for nonlinear part}: Characterizing the continuity of nonlinear maps in the weak topology is often more challenging. Thus, a higher regularity requirement is imposed instead. Specifically, we desire $u_N^3$ converges strongly to $u^3$ in $L^{4/3}(\Omega)\hookrightarrow H^{-1}(\Omega)=\BT$ if $u_N$ converges weakly to $u$ in $\UT$. We then observe that:
    $$\begin{aligned}
    \|u_N^3-u^3\|_{L^{4/3}(\Omega)}^{4/3}&=\int_\Omega (u_N-u)^{4/3}(u_N^2+uu_N+u^2)^{4/3}dx\\
    &\le \left(\int_\Omega(u_N-u)^4dx\right)^{1/3}\left(\int(u_N^2+uu_N+u^2)^2dx\right)^{2/3}\\
    &\le 2\|u_N-u\|^{4/3}_{L^4}\left(\|u_N\|^{8/3}_{{L^4}}+\|u\|^{8/3}_{{L^4}}\right),
    \end{aligned}$$
    which means that  $\UT$ \textbf{compactly embedded} into $L^4(\Omega)$ is sufficient, and this is a direct consequence of $\UT$ continuously embeds into $H^1_0(\Omega)$. 
\end{itemize}
However, Proposition \ref{OCNonlinear} clearly indicates that the nonlinear optimal recovery does not yield a finite-dimensional solution. Therefore, to render our approach applicable, additional steps towards a relaxed nonlinear optimal recovery must be taken.

\subsubsection{Relaxed Nonlinear Optimal Recovery}
\label{subsubsec:QNOR}
We begin by verifying that our equation conforms to the general form \eqref{generaloperator}. Noticing that if we take $L_1=\Delta$, $L_2=\text{id}$ and $F(x,y)=-x+y^3$, we have:
$$ F(L_1u(x),L_2u(x))=-\Delta u(x)+u^3(x) $$
Following the analysis presented in Section \ref{sec:NEinFS}, we employ a point evaluation approximations $\{\sum_{m=1}^M c_{mn}\delta_{x_{mn}}\}$ of the test functions $\{ \phi_n(x)\}$ to achieve a finite-dimensional solution. Then the corresponding scale $M$ Relaxed Nonlinear optimal recovery is
\begin{equation}
\label{quasiNFORPDE}
\begin{aligned}
\min_{u\in\mathcal{U}}&\ \ \ \ \ \ \ \ \ \|u\|_\mathcal{U}\\
s.t&\ \ \left| \sum_{m=1}^Mc_{nm}\left(-\Delta u(x_{mn})+u^3(x_{mn})\right)-\int_\Omega f(x)\phi_n(x)dx\right|\le \epsilon_{nM},\ n=1,...,N
\end{aligned}
\end{equation}
To evaluate the nonlinear PDE at specific points, we require that $-\Delta u + u^3$ be pointwise bounded. This necessitates a sufficient condition that our solution space $\mathcal{U}$ is continuously embedded in $C^2(\Omega)$. 

According to the Sobolev embedding theorem, we see that one possible choice of $\UT$ is $H^4(\Omega)$. In this case, $\BT = H^2(\Omega)$ is a reproducing kernel Hilbert space (therefore, a reflexive reproducing kernel Banach space).

Since we are using the point evaluation approximation, the assumption \ref{assu3} is satisfied automatically. Consequently, we establish the relaxed nonlinear optimal recovery convergence as formulated in \eqref{quasiNFORPDE}. In conclusion, if the solution of equation (\ref{PDEexampleD}) belongs to $H^4(\Omega)$ and we choose $H^{-2}(\Omega)$ as the test function space, then the solution of the relaxed nonlinear optimal recovery problem (\ref{quasiNFORPDE}) will converge to the true solution.

Furthermore, from Proposition \ref{OCQuasiFunction}, we can conclude that our relaxed nonlinear optimal recovery \eqref{quasiNFORPDE} admits a finite-dimensional solution:
$$u_N^M=\sum_{n=1}^N\sum_{m=1}^M\lambda_{mn}\psi_{mn}+\hat{\lambda}_{mn}\hat{\psi}_{mn}$$
where $\psi_{mn}$ and $\hat{\psi}_{mn}$ are Riesz representer satisfied $-\Delta u(x_{mn})=\langle u,\psi_{mn}\rangle_\UT$ and $u(x_{mn})=\langle u,\hat{\psi}_{mn}\rangle_\UT$
\subsection{Some Generalizations}
\label{subsec:MOGene}
In this subsection, we will illustrate how the linear-nonlinear decomposition introduced in Subsection \ref{subsec:OD} enables us to relax the regularity requirements in relaxed nonlinear optimal recovery. Additionally, we show that the multi-domain generalization presented in Subsection \ref{subsec:MDG} facilitates handling cases involving Robin boundary conditions.
\subsubsection{Reduction of Regularity Requirement}
\label{subsubsec:RoRR}
Although the relaxed nonlinear optimal recovery solution of \eqref{PDEexampleD} presented in Subsection \ref{subsubsec:QNOR} is finite-dimensional, it requires that $\mathcal{U}$ be continuously embedded in $C^2(\Omega)$, which imposes a very high regularity requirement in classical PDE theory. However, by employing the operator decomposition technique introduced in Subsection \ref{subsec:OD}, we can reduce this regularity requirement while preserving the finite-dimensional property of the solution. Specifically, by employing a point evaluation approximation point evaluation approximations $\{\sum_{m=1}^M c_{mn}\delta_{x_{mn}}\}$ of the test functions $\{ \phi_n(x)\}$  we propose the following formulation:
\begin{equation}
\label{quasiNFORPDELN}
\begin{aligned}
\min_{u\in\mathcal{U}}&\ \ \ \ \ \ \ \ \ \|u\|_\mathcal{U}\\
s.t&\ \ \left| -\int_\Omega \nabla u(x)\cdot\nabla\phi_n(x)dx+\sum_{m=1}^Mc_{mn}u^3(x_{mn})-\int_\Omega f(x)\phi_n(x)dx\right|\le \epsilon_{nM},\ n=1,...,N
\end{aligned}
\end{equation}
Notably, in this formulation, we only require $u^3$ to be pointwise bounded. Consequently, the regularity requirement for $u$ is reduced to $\mathcal{U}$ being embedded in $H^1_0(\Omega) \cap C(\Omega)$. This represents a significant reduction in regularity compared to the previous $C^2(\Omega)$ requirement. 

By definitions in Subsection \ref{subsec:OD}, we have $\BT_L=H^{-1}(\Omega)$ and $\BT_N$ is continuously embedded into $C^0(\Omega)$. By the Sobolev embedding theory, we see that one possible choice is $\BT_N=H_0^2(\Omega)$, which is both reflexive and reproducing. In this case since $H^1_0(\Omega)=H^{-1}(\Omega)\hookrightarrow (H_0^2(\Omega))^*$, the test function space is $\BT^*=H^1_0(\Omega)$ and $\BT = H^{-1}(\Omega)$.

Furthermore, following the discussion about the equation  \eqref{Linear-Nonlinear Case}, we can conclude that the solution $u_N$ of the equation \eqref{quasiNFORPDELN} is also finite-dimensional and satisfies:
$$ u^M_N=\sum_{n=1}^N\lambda_n\psi_n+\sum_{n=1}^N\sum_{m=1}^M\hat{\lambda}_{mn}\psi_{mn} $$
where $\psi_n,\psi_{n,q}$ are given by Riesz representation theorem, satisfying
$$ -\int\nabla u(x)\cdot\nabla\phi_n(x)dx=\langle u,\psi_n\rangle_\UT\,\quad\text{ and }\quad u(x_{mn})=\langle\psi_{mn},v\rangle_\UT.$$

\subsubsection{The Case of Robin Boundary Conditions}
In this section, we look into the case of Robin Boundary condition of the PDE \eqref{PDEexample}. Namely, we consider the following PDE systems:
\begin{equation}
\label{PDEexampleR}
    \left\{\begin{aligned}
    -\Delta u(x)+u^3(x)=f(x),\ \ x\in\Omega \\
    u(x)+\frac{\partial u}{\partial n}(x)=0\ \ ,\ \ x\in\partial\Omega .
\end{aligned}\right.
\end{equation}
We use a multi-domain representation involving $\Omega$ and $\partial \Omega$ 
with the test function space being a product space $\BT=H^1(\Omega)\times L^2(\partial\Omega)$,
then the weak form of this PDE is given by:
$$\int_\Omega \nabla u(x)\cdot\nabla \phi(x)+u^3(x)\phi(x)dx+{\int_{\partial\Omega}u(x)\phi(x)dx}=\int_\Omega f(x)\phi(x)dx,\ \ \ \ \forall\phi\in \BT$$
and the solution space derived from classical PDE theory is $\UT \hookrightarrow  \BT=H^1(\Omega)\times L^2(\partial\Omega)$. Now we can reformulate our nonlinear optimal recovery \eqref{NOR} as follows:
\begin{equation}
\label{NFORPDER}
\begin{aligned}
\min_{u\in\mathcal{U}}&\ \ \ \ \ \ \ \ \ \|u\|_\mathcal{U}\\
s.t&\ \ \int_\Omega \nabla u(x)\cdot\nabla \phi(x)+u^3(x)\phi(x)dx+{\int_{\partial\Omega}u(x)\phi(x)dx}=\int_\Omega f(x)\phi(x)dx,\ \ n=1,...,N
\end{aligned}
\end{equation}
To make the convergence theorem\ref{convergeNOR} work for this nonlinear optimal recovery, we have to verify the following requirement:
\begin{itemize}
    \item \textbf{Weak continuity (on the boundary $\partial\Omega$}): As the mapping is identical on the boundary $\partial\Omega$, the weak continuity requirement is automatically satisfied.
\end{itemize}
Therefore theorem~\ref{convergeNOR} holds true for the equation $\eqref{NFORPDER}$ and $\UT = H^1(\Omega)\times L^2(\partial\Omega)$.

Furthermore, by setting $\Omega_1=\Omega$ and $\Omega_2=\partial\Omega$, and defining the subdomain operators as $\mathcal{F}_1(u)=-\Delta u+u^3$ and $\mathcal{F}_2(u)=u$, we can reformulate the problem to align with the framework presented in Section \ref{subsec:MDG} as follows:
\begin{eqnarray*}
 [\FT_1(u),\phi]+[\FT_2(u),\phi] &=& \int_\Omega \nabla u(x)\cdot\nabla \phi(x)+u^3(x)\phi(x)dx+{\int_{\partial\Omega}u(x)\phi(x)dx}\\
 &=&  \int_\Omega f(x)\phi(x)dx.
\end{eqnarray*}

Applying the linear-nonlinear decomposition discussed in the previous section on $\Omega_1$, we implement the point evaluation approximations point evaluation approximations $\{\sum_{m=1}^M c_{mn}\delta_{x_{mn}}\}$ of the test functions $\{ \phi_n(x)\}$ solely for the nonlinear term $u^3$. This approach yields the relaxed nonlinear optimal recovery:
\begin{equation}
\label{quasiNFORPDELNR}
\begin{aligned}
\min_{u\in\mathcal{U}}&\ \ \ \ \ \ \ \ \ \|u\|_\mathcal{U}\\
s.t&\ \ \left| -\int_\Omega \nabla u(x)\cdot\nabla\phi_n(x)dx+\sum_{m=1}^Mc_{mn}u^3(x_{mn})\right.\\
& \qquad\qquad \left. +{\int_{\partial\Omega}u(x)\phi(x)dx}-\int_\Omega f(x)\phi_n(x)dx\right|\le \epsilon_{nM},\ n=1,...,N.
\end{aligned}
\end{equation}
Following similar discussion in subsection \ref{subsubsec:RoRR}, we take $\UT = H^2(\Omega)\times L^2(\partial\Omega)$.

From the discussion of subsection \ref{subsec:MDG}, we have the solution $u^M_N$ satisfies:
$$u_N^M=\sum_{n=1}^N\lambda_n\psi_n+\sum_{n=1}^N\hat{\lambda}\hat{\psi}_n+\sum_{n=1}^N\sum_{m=1}^M\lambda_{mn}\psi_{mn}$$
where $\psi_n,\psi_{n,q},\hat{\psi}_n$ are given by Riesz representation theorem, satisfying 
$$ -\int_\Omega\nabla u(x)\cdot\nabla\phi_n(x)dx=\langle u,\psi_n\rangle_\UT ,\;\; u(x_{mn})=\langle\psi_{mn},v\rangle_\UT \;\;\text{ and }\;\;\int_{\partial\Omega}u(x)\phi_n(x)dx=\langle u,\psi_n\rangle_\UT.$$

\section{Conclusion and Future Direction}
\label{sec:CandFD}
This paper introduces a general framework called nonlinear optimal recovery for solving nonlinear problems with finite measurements. Our theoretical analysis demonstrates that as the number of measurements increases, the solution derived from this method converges to the true solution of the problem. 
Moreover, we observe that the solution structure is generally infinite-dimensional. To 
offer further theoretical insight, we propose a sufficient condition for achieving finite-dimensional solutions.

To accommodate a broader range of cases that might arise in applications, we develop a relaxed nonlinear optimal recovery and provide a comprehensive convergence analysis. For nonlinear equations in infinite dimension, we present key ingredients to assure the finite-dimensional condition and formulate the corresponding model. Additionally, we extend our framework to scenarios where it might be convenient to formulate the problem using multiple operators.
In addition, to validate our models and theoretical findings, we examine them through an illustrative example: the steady-state equation of a reaction-diffusion system.

There are numerous potential future research directions stemming from this work, such as those outlined below.
\begin{itemize}
    \item \textbf{Sequential Optimal Recovery}: As a generalization of our nonlinear optimal recovery, We propose the following \textbf{sequential optimal recovery}. Consider a space sequence $\{\UT_n\}_{n=1}^\infty$ such that $\UT_1\subset\UT_2\subset\cdots \subset \UT_n\subset\cdots 
    $, the sequential optimal recovery is
$$\begin{aligned}
\min_{u\in\mathcal{U}_N}&\ \ \ \ \ \ \ \ \ \|u\|_\mathcal{U}\\
s.t&\ \ [\FT(u),\phi_n]=[f,\phi_n],\ \ n=1,...,N
\end{aligned}$$
It should be recognized that this formulation offers a broadly applicable framework that can cover many different cases.
\begin{itemize}
    \item \textbf{Nonlinear Optimal Recovery}: When we take $\UT_N=\UT$, this is exactly our nonlinear optimal recovery problem
    \item \textbf{Galerkin/Petrov Galerkin}: When we take $\UT_N=\text{Span}\{\phi_n\}$, it becomes the 
    Galerkin method. More generally, if we take $\UT_N$ be an $M_N$ dimensional subspace of $\UT$, then it coincides the 
    Petrov Galerkin methods.
\end{itemize}
    In earlier discussions on models like the nonlinear optimal recovery \eqref{NOR}, regularization model \eqref{regularization}, and relaxed nonlinear optimal recovery \eqref{quasiNOR}, we considered the original solution space $\UT$.
    Such a choice made it challenging to obtain a finite-dimensional solution. However, with the extension to a sequential optimal recovery formulation, we allow the solution space to change as the number of measurements changes, making it easier to obtain a finite-dimensional solution. For instance, one can choose the solution space $\UT_n$ to be finite-dimensional, thereby ensuring that the solution is automatically finite-dimensional. The conventional Galerkin and Petrov-Galerkin approximations can be seen as special cases.

    Obviously, the convergence theorem discussed in Section \ref{subsec:NOR} cannot be directly extended to arbitrary sequential optimal recovery problems. This is because the true solution may not belong to the subspace $\mathcal{U}_n$. Thus, generalizing the convergence analysis for sequential optimal recovery presents an intriguing avenue for further investigation.
    \item \textbf{Generalization to Banach Space}
The present work is built within the context of a Hilbert solution space. However, for many PDE problems, the most natural solution space might not be a Hilbert space, but rather a Banach space. When the solution space is a Banach Space $\tilde{\BT}$, the optimal recovery problem becomes:
    $$\begin{aligned}
    \min_{u\in \tilde{\BT}}&\ \ \ \ \ \ \ \ \ \|u\|_{\tilde{\BT}} \\
    s.t&\ \ [\FT(u),\phi_n]=[f,\phi_n],\ \ n=1,...,N
    \end{aligned}$$

    By computing the Fr\'{e}chet derivative of the Lagrangian, we obtain the condition for the solution $u_N$:
    $$u_N=R_{\tilde{\BT}}^{-1}\left(\sum_{n=1}^N\lambda_n[D_u\FT(u)\cdot,\phi_n]\right)$$
    where $R_{\tilde{\BT}}$ is the duality mapping on $\tilde{\BT}$, which is generally nonlinear. Due to the nonlinearity of this mapping, our study on finite-dimensional solutions cannot be directly extended to this case. Therefore, another interesting avenue for future research is to generalize our formulation and the theory of nonlinear optimal recovery to encompass Banach spaces as the underlying solution space.
    \item \textbf{Error Analysis}: The discussion here on the nonlinear optimal recovery is based on the minimum regularity requirements, precluding a meaningful error estimation analysis that usually requires more strict regularity assumptions. Further research into the rate of convergence and/or other asymptotic error estimates would be a valuable avenue to explore.
\end{itemize}

Finally, we note that our work here has focused on theoretical development. There are certainly many practical issues, such as the computational efficiency, to be investigated in future research. Also, the accessibility to measurement data is a key part of the optimal recovery. It will be interesting to consider optimal design problems to find out what are the most theoretically desirable (and practically observable) measurement data in order to achieve better solutions for specific problems from applications.

\section*{Acknowledgements}
This research is supported in part by NSF-DMS 2309245,  
DE-SC0022317 and
DE-SC0025347.

\section{appendices}
\label{sec:append}
\subsection{Proof of Theorem \ref{convergeNOR}}
\label{proofThm1}
\begin{proof}

    We first show the existence of a minimizer for the nonlinear optimal recovery problem \eqref{NOR} at scale $n$, followed by a convergence of the minimizer sequence to the true solution of the nonlinear equation \eqref{nonlinearproblem} in the strong topology of the Hilbert space $\mathcal{U}$.

    \textbf{1. Existence of Minimizer}. Since our objective function is the Hilbert space norm,  it is bounded below by 0. Therefore, there exists a minimizing sequence $\{u^M_N\}_{M=1}^\infty$. Also,  the minimizing sequence is actually minimizing the Hilbert norm. It is thus a bounded sequence in $\mathcal{U}$.

    By the weak compactness of the Hilbert space, we can then find a weakly convergent subsequence of $\{u^M_N\}$. With a little abuse of notation, we will still denote this subsequence as $\{u^M_N\}$ and the weak limit as $\tilde{u}_N$. We then show that $\tilde{u}_N$ satisfied all the constraints in the optimal recovery problem in scale $N$. Noticing that by our second assumption, we can show that:
    $$[\FT(u^M_N),\phi_n]\rightarrow [\FT(\tilde{u}_N),\phi_n],$$
    for $n=1,...,N$. Therefore we will have the weak limit satisfy all the constraints. Moreover, by the weak lower semi-continuity of  norms, we have:
$$\|\tilde{u}_N\|_\mathcal{U}\le\liminf_{n\rightarrow\infty}\|u^M_N\|_{\mathcal{U}}$$
    Since $\{u^M_N\}$ is the minimizing sequence,  we can conclude that $\tilde{u}_N$ is the minimizer of the optimal recovery problem \eqref{NOR}.

    \textbf{2. Convergence Analysis}. Since we have $\|u^*\|_\mathcal{U}\ge \|u_N\|_\mathcal{U}$,  our solutions series $\{u_N\}$ is a bounded series in $\UT$. Therefore we can find a weakly convergent subsequence with weak limit $\hat{u}$. With a little abuse of notation, we still denote the subsequence as $\{u_N\}$ (Note that, as shown later, $\hat{u}$ is actually the weak limit for the whole sequence).

    We wish to prove that $\hat{u}$ satisfies the nonlinear equation  \eqref{nonlinearproblem}. WLOG, we assume that $f$=0. Then our equation becomes:
    $$[ \FT(u),\phi]=0\ \ ,\ \ \forall \phi\in \BT^*$$

    Now, for any $\tilde{\phi}_N\in \text{Span}\left(\phi_1,...,\phi_N\right)$, by summing the constraints we have for $M\ge N$:
    $$[\FT(u_M),\tilde{\phi}_N]=0$$
    by the weak convergence of $\FT(u_M)$(condition 2), we have as $M\rightarrow\infty$, that
    $$[\FT(u_M),\tilde{\phi}_N]\rightarrow [\FT(\hat{u}),\tilde{\phi}_N].$$
    Therefore we can conclude that:
    $$[\FT(\hat{u}),\tilde{\phi}_N]=0.$$
    
    By the first condition of our theorem, we can find a sequence of functions  $\{\hat{\phi}_N(x)\}$ that converges weakly to $\phi$ in $\BT^*$. Therefore we have:
$$[\FT(\hat{u}),\hat{\phi}_N]\rightarrow [\FT(\hat{u}),\phi].$$
     We can then conclude that
    $$[\FT(\hat{u}),\phi]=0,$$
    This means that the weak limit $\hat{u}$ satisfies the nonlinear equation \eqref{nonlinearproblem}. By the uniqueness of the solution, we have $u^*=\hat{u}$. Also, since the weak limit is independent of the sequence,  we conclude that $u$ is actually the weak limit of the whole sequence in $\UT$

    Next, we prove that $u$ is actually the strong limit in $\mathcal{U}$. Noticing that norm convergence and weak convergence imply strong convergence. With the already established weak convergence, we only need to prove the norm convergence. To do this, notice that the Hilbert space norm is a convex functional and weakly lower semi-continuous. As $u$ is the weak limit of $\{u_N\}$, we get
$$\|u^*\|_\mathcal{U}\le\liminf_{n\rightarrow\infty}\|u_N\|_{\mathcal{U}}.$$
    Also, by the definition of the optimal recovery problem, $u_N$ is the smallest norm solution with $N$ measurement, so we also have $\|u^*\|_\mathcal{U}\ge\|u_N\|_\mathcal{U}$ for any $N>0$. Therefore we conclude that:
$$\lim_{N\rightarrow\infty}\|u_N\|_\mathcal{U}=\|u^*\|_\mathcal{U}.$$
    This shows that $u^*$ is actually the strong limit of $u_N$ in $\UT$.
\end{proof}
\subsection{Proof of Theorem \ref{convergeregu}}
\begin{proof}
    Similar to part 1 in Theorem \ref{convergeNOR}, we can show the existence of the solution $u^\mu_N$. We thus focus on proving the two convergence results.

    \textbf{1. Convergence to $u_N$}: We define
$$R_N(u)=\sum_{n=1}^N([\FT(u),\phi_n]-[f,\phi_n])^2$$
    For any $\mu>0$, we observe the following inequality:
    $$\|u_N^\mu\|_\UT\le \|u_N^\mu\|_\UT+\mu R_N(u_N^\mu)\le \|u_N\|_\UT+\mu R_N(u_N)= \|u_N\|_\UT.$$
Therefore, for any sequence $\{\mu_m\}$ such that $\mu_m\rightarrow\infty$ as $m\rightarrow\infty$,  the sequence $\{u^{\mu_m}_N\}$ is bounded.
So, there is a weakly convergent subsequence $\{u^{\mu_m}_N\}$(same as Theorem \ref{convergeNOR}, a little abuse of notations) with a weak limits $\hat{u}_N$. By the weak continuity of the $\FT$, we have $$R_N(u^{\mu_m}_N)\rightarrow R_N(\hat{u}_N)$$
Also, noticing that:
    $$ R_N(u^{\mu_m}_N)\le \frac{1}{\mu_m}\|u^{\mu_m}_N\|_\UT+R_N(u^{\mu_m}_N)\le  \frac{1}{\mu_m}\|u_N\|_\UT\rightarrow 0,$$
    we  conclude that $R_N(\hat{u}_N)=0$. Moreover, by the weak lower semi-continuity of the Hilbert space norm, we have:
$$\|\hat{u}_N\|\le\liminf_{m\rightarrow\infty} \|u_N^{\mu_m}\|_\UT\le \|u_N\|_\UT.$$
    Since $u_N$ is the minimum norm element with the constraints $R_N(u)=0$, we get $\hat{u}_N=u_N$ and it is the weak limits of the whole sequence $\{u^{\mu_m}_N\}$. 
    
    Now, it remains to prove the norm convergence.  Noticing that
$$\|u_N\|\le\liminf_{m\rightarrow\infty}\|u_N^{\mu_m}\|_\UT\le\limsup_{m\rightarrow\infty} \|u_N^{\mu_m}\|_\UT\le  \|u_N\|_\UT,$$
    we conclude that $\lim_{m\rightarrow\infty}\|u^{\mu_m}_N\|=\|u_N\|$. Since the series $\{\mu_m\}$ is arbitrary, we have $u^\mu_N\rightarrow u$ strongly in $\UT$.

    \textbf{2. Convergence to $u$}: For any $\mu>0$, we observe the following inequality
$$\|u_N^\mu\|_\UT\le \|u_N^\mu\|_\UT+\mu R_N(u_N^\mu)\le \|u^*\|_\UT+\mu R_N(u^*)= \|u^*\|_\UT.$$
For any sequence $\{\mu_N\}$ such that $\mu_N\rightarrow\infty$ as $N\rightarrow\infty$,  the sequence $\{u^{\mu_N}_N\}$ is a bounded sequence.  Therefore, there is a weakly convergence subsequence $\{u^{\mu_N}_N\}$(same as Theorem \ref{convergeNOR}, a little abuse of notations) with a weak limit $\hat{u}$. Now for any $\phi\in \BT^*$, by the assumptions, we can find a sequence $\{\hat{\phi}_M\}$ such that $\hat{\phi}_M\in\text{Span}\{\phi_1,...,\phi_M\}$ and
$$ [\FT(\hat{u}),\hat{\phi}_M]\rightarrow [\FT(\hat{u}),\phi]\ \ \ as\ M\rightarrow\infty.$$
Similar to the proof in Theorem\ref{convergeNOR}, one can get that
$$[\FT(u^{\mu_N}_N),\hat{\phi}_M]\rightarrow[\FT(\hat{u}),\hat{\phi}_M]\ \ \ as\ N\rightarrow\infty.$$
Also, as $N\rightarrow \infty$, we have:
$$R_N(u^{\mu_N}_N)\le \mu_N\|u^{\mu_N}_N\|_\UT+R_N(u^{\mu_N}_N)\le \mu_N\|u^*\|_\UT\rightarrow 0,$$
therefore $ [\FT(u^{\mu_N}_N), \hat{\phi}_M]\rightarrow 0 $ as $N\rightarrow\infty$. So we conclude that:
$$[\FT(\hat{u}),\phi]=0$$ which means $\hat{u}=u^*$. Next, we show the norm convergence. By the weak lower semi-continuity of the norm, we have:
$$\|u^*\|_\UT\le \liminf_{N\rightarrow\infty}\| u^{\mu_N}_N \|_\UT $$
and for any $\mu_N>0$
$$\| u^{\mu_N}_N \|_\UT\le\| u^{\mu_N}_N \|_\UT+\mu_N R_N(u^{\mu_N}_N)\le \| u^* \|_\UT+\mu_N R_N(u^*)= \| u^* \|_\UT$$
therefore $\lim_{N\rightarrow\infty} \| u^{\mu_N}_N \|_\UT=\| u^* \|_\UT$. By the arbitrary of the series $\{\mu^N\}$, we have $u^\mu_N\rightarrow u^*$ as $(\mu,N)\rightarrow(\infty,\infty)$ .
\end{proof}

\subsection{Proof of Theorem \ref{ConvergeQNOR}}
\begin{proof}
     Similar to part 1 in Theorem \ref{convergeNOR}, we can show the existence of the solution $u^\mu_N$. Therefore we will focus on proving 2 convergence results.

    \textbf{1. Convergence to $u_N$}: Take the $\{\delta_{nM}\}$ for $u_N$ in lemma \ref{lemmafeasible} we have the true solution $u$ is in the feasible set of the $M-$points approximation algorithm. Therefore the solution for the relaxed nonlinear optimal recovery \eqref{quasiNOR} $\{u^M_N\}_{M=1}^\infty$ is a bounded sequence is $\UT$ since we have:
    $$\|u^M_N\|_\mathcal{U}\le \|u\|_\mathcal{U}.$$Therefore it will have a weak convergent sub-sequence with a weak limit $\hat{u}_N$ when $M\rightarrow\infty$ in $\mathcal{U}$. With a little abuse of notation, we will still denote the weak convergent sub-sequence as $\{u^M_N\}_{M=1}^\infty$, and we will show that $\hat{u}_N$ is actually the strong limit for the whole series later.

     Next, We show that the weak limit $\hat{u}_N$ is actually the solution to the optimal recovery problem \eqref{NOR}. First, for any $n=1,...,N$, since we have $u^M_N$ converges weakly to $\hat{u}_N$ in $\mathcal{U}$ and $\FT(u^M_N)$ converges weakly to $\FT(\hat{u}_N)$ in $\BT$. Therefore there exist a sequence of constants $\{\epsilon^N_{M}\}$ goes to 0 and such that
     $$\left|\left[\FT(u^M_N),\phi_n\right]-\left[\FT(\hat{u}_N),\phi_n\right]\right|\le \epsilon^N_{M}$$
     where the  symbol $[\cdot,\cdot]$ is the duality pairing. Notice that from lemma\ref{lemmafeasible}, we have $\|u^M_N\|_\UT\le \|u^*\|_\UT$, therefore
     $$\left|\left[\FT(u^M_N),\phi_n-\sum_{m=1}^M c_{nm}\phi_{mn}\right]\right|\le \epsilon_{nM}$$

    Next, WLOG we assume $f=0$, Then we have for any $M>0$
    \begin{equation*}
    \begin{aligned}
     \left|\left[\FT(\hat{u}_N),\phi_n\right]\right| &\le \left|\left[\FT(u^M_N),\phi_n\right]-\left[\FT(\hat{u}_N),\phi_n\right]\right|+\left|\left[\FT(u^M_N),\phi_n\right]\right|\\
    &\le \epsilon_M^N+\left|\left[\FT(u^M_N),\sum_{m=1}^M c_{nm}\phi_{mn}\right]\right|+\left|\left[\FT(u^M_N),\phi_n-\sum_{m=1}^M c_{nm}\phi_{mn}\right]\right|\\
    &\le \epsilon_M^N+2\epsilon_{nM}
    \end{aligned}
\end{equation*}
Therefore fixing a $N$, when $M\rightarrow\infty$, we can conclude that:
$$\left[\FT(\hat{u}_N),\phi_n\right]=0\ \ ,\ \ n=1,...,N$$
So it satisfied all the constraints for our optimal recovery problem \eqref{NOR}. Next, we will verify the minimum norm property. We will prove by contradiction, suppose $\|u_N\|_\mathcal{U}<\|\hat{u}_N\|_\mathcal{U}$. Since $\hat{u}_N$ is the weak limit of $\{u^M_N\}_{M=1}^\infty$. By the weakly lower semi-continuous of the norm we have $\|\hat{u}_N\|_\mathcal{U}\le \lim\inf \|u^M_N\|_\mathcal{U}$, but since $u_N$ is in the feasible set of the problem  \eqref{quasiNOR} so we should have $\|u_N\|_\mathcal{U}\ge\|u^M_N\|_\mathcal{U}$. So we should have $\|u_N\|_\mathcal{U}\ge\|\hat{u}_N\|_\mathcal{U}$ , that is a contradiction. By assuming the solution for  \eqref{NOR} is unique, then we have $\hat{u}_N=u_N$. Since the weak limit $u_N$ is independent of the sequence, it is actually the limit of the whole series.

Lastly, we will show that $u_N$ is actually the strong limit. First by the weakly lower semi-continuous of norm we have $\|u_N\|_\mathcal{U}\le \lim\inf \|u^M_N\|_\mathcal{U}$. Also since $u_N$ is in the feasible set of  \eqref{quasiNOR}, we should have $\|u_N\|_\mathcal{U}\ge\|u^M_N\|_\mathcal{U}$. Therefore we can conclude that:
$$\|u_N\|=\lim_{n\rightarrow\infty}\|u^M_N\|$$
By weak convergence and norm convergence implies strong convergence, $u_N$ is actually the strong limit for $\{u^N_M\}$.

\textbf{2. Convergence to $u$}: Take the $\{\delta_{nNM}\}$ for $u^*$ in lemma\ref{lemmafeasible}. Firstly, for any function series $\{u^M_N\}$ with $(M,N)\rightarrow(\infty,\infty)$, since it is a bounded series, it has a weak convergent subsequence with a weak limit $\tilde{u}$. Next, we will try to show that:
$$ \left[\FT(\tilde{u}),\phi_n\right]=[f,\phi]\ \ ,\ \ \forall \phi\in \BT^*$$

WLOG, we assume $f=0$, we only need to prove for any fixed $\phi\in \BT^*$ and $\epsilon>0$, we have:
$$\left[\FT(\tilde{u}),\phi_n\right]\le \epsilon$$

First we will pick a $H>0$ such that there exist a $\tilde{\phi}_H=\sum_{h=1}^Hc_h\phi_h$ such that:
$$\textbf{(Assumption \ref{assu1}) }\left| [\FT(\tilde{u}),\phi]-[\FT(\tilde{u}),\tilde{\phi}_H]\right|\le \frac{\epsilon}{4}$$

Next, let $C_H=\sum_h |c_l|$ we pick a $(N,M)$ pair in the series large enough such that $N>H$ and:
$$\textbf{(Assumption \ref{assu11})  }\left|\left[\FT(u^M_N),\phi_h\right]-\left[\FT(\tilde{u}),\phi_h\right]\right| \le\frac{\epsilon}{4C_H}\ \ ,\ \ h=1,...,H$$
$$ \textbf{(Assumption \ref{assu3})  }\left|\left[\FT(\tilde{u}),\phi_h-\sum_{m=1}^M c_{mh}\phi_{mh}\right]\right|\le \frac{\epsilon}{4C_H}\ \ ,\ \ h=1,...,H$$
$$\textbf{(Algorithm assumption) }\left| \left[\FT(u^M_N),\sum_{m=1}^M c_{mh}\phi_{mh}\right]\right|\le \frac{\epsilon}{4C_H},\ h=1,...,H$$

Then we have:
\begin{equation*}
    \begin{aligned}
        &\left| [\FT(\tilde{u})-\FT(u^M_N),\tilde{\phi}_H]\right|\le \left|\left[\FT(u^M_N),\tilde{\phi}_H\right]-\left[\FT(\tilde{u}),\tilde{\phi}_H\right]\right|\\
        &\le\sum_{h=1}^H|c_h|\left(\left|\left[\FT(u^M_N),\phi_h\right]-\left[\FT(\tilde{u}),\phi_h\right]\right|\right)\le \frac{\epsilon}{4}
    \end{aligned}
\end{equation*}
Also,
\begin{equation*}
    \begin{aligned}
    \left|\left[\FT(\tilde{u}),\sum_{h=1}^Hc_h\left(\sum_{m=1}^M c_{mh} \phi_{mh}\right)-\tilde{\phi}_H\right]\right|
    \le \sum_{h=1}^H|c_h|\left|\left[\FT(\tilde{u}),\phi_h-\sum_{m=1}^M c_{mh}\phi_{mh}\right]\right|\le \frac{\epsilon}{4}
    \end{aligned}
\end{equation*}
And,
\begin{equation*}
    \begin{aligned}
    \left|\left[\FT(\tilde{u}),\sum_{h=1}^Hc_h\left(\sum_{m=1}^M c_{mh} \phi_{mh}\right)\right]\right|\le \sum_{h=1}^H |c_{h}|\left|\left[\FT(\tilde{u}),\sum_{m=1}^M c_{mh} \phi_{mh}\right]\right|\le\frac{\epsilon}{4}
    \end{aligned}
\end{equation*}

Summing these 4 inequalities together we will get
$$ \left|[\FT(\tilde{u}_N),\phi]-[f,\phi]\right|\le\epsilon\ \ ,\ \ \forall \phi\in \BT^*$$
for any $\epsilon>0$, so by the uniqueness of the solution of the equation, we have $\tilde{u}=u$. Utilizing similar arguments in part 1 we can show that $\tilde{u}$ is also the strong limit for $\{u^M_N\}$, so we complete the proof.
\end{proof}

\bibliography{sn-bibliography}
\end{document}